\documentclass[10pt,a4paper]{amsart}

\usepackage{amssymb}
\usepackage[utf8]{inputenc}
\usepackage[english]{babel}
\usepackage{amsfonts,mathtools}
\usepackage{amsmath}
\usepackage{comment}
\usepackage{transparent}
\usepackage{graphicx}
\usepackage{xcolor}
\usepackage[normalem]{ulem}
\usepackage{color}
\usepackage{upgreek}

\DeclareFontFamily{U}{mathx}{\hyphenchar\font45}
\DeclareFontShape{U}{mathx}{m}{n}{
      <5> <6> <7> <8> <9> <10>
      <10.95> <12> <14.4> <17.28> <20.74> <24.88>
      mathx10
      }{}
\DeclareSymbolFont{mathx}{U}{mathx}{m}{n}
\DeclareFontSubstitution{U}{mathx}{m}{n}
\DeclareMathAccent{\widecheck}{0}{mathx}{"71}
\DeclareMathAccent{\wideparen}{0}{mathx}{"75}

\usepackage[bookmarksopen,bookmarksdepth=3,colorlinks,citecolor=red,pagebackref,hypertexnames=false]{hyperref}
\usepackage{amsthm}
\usepackage[nameinlink]{cleveref}

\usepackage{enumitem}
\usepackage{mathabx}

\newtheorem{main-theorem}{Theorem}
\newtheorem{proposition}{Proposition}[section]
\newtheorem{theorem}[proposition]{Theorem}

\newtheorem{lemma}[proposition]{Lemma}
\theoremstyle{remark}
\newtheorem{remark}[proposition]{Remark}

\theoremstyle{definition}

\newtheorem*{acknowledgements}{Acknowledgements}

\DeclareMathOperator{\supp}{supp}

\DeclareMathOperator*{\esssup}{ess\,sup}

\newcommand{\R}{\mathbb{R}}
\newcommand{\C}{\mathbb{C}}

\newcommand{\N}{\mathbb{N}}
\newcommand{\Sph}{\mathbb{S}}
\newcommand{\dd}{\mathrm{d}}
\newcommand{\tm}{\mathrm{t}}
\newcommand{\x}{\mathrm{x}}
\newcommand{\Id}{\mathrm{Id}}

\newcommand*{\defeq}{\mathrel{\vcenter{\baselineskip0.5ex \lineskiplimit0pt
			\hbox{\scriptsize.}\hbox{\scriptsize.}}}%
	=}

\newcommand{\X}{\mathrm{X}}
\newcommand{\dist}{\mathrm{dist}}
\newcommand{\om}{\omega}

\def\Im{\mathop\mathrm{Im}}
\def\Re{\mathop\mathrm{Re}}
\def\loc{\mathop\mathrm{loc}}
\def\ve{\varepsilon}
\def\lec{\lesssim}
\def\vp{\varphi}
\def\wh{\widehat}

\def\cal{\mathcal}
\def\ovl{\overline}
\def\cd{\centerdot}
\def\wt{\widetilde}
\def\d{\partial}
\def\bs{\boldsymbol}

\title[Stability initial to final]{Stability estimates for the initial-to-final-state inverse problem}

\author[M. Cañizares]{Manuel Cañizares}
\address[M.\ Cañizares]{ RICAM -The Johann Radon Institute for Computational and Applied Mathematics, 
Altenberger Str. 69, 4040 Linz, Austria,}
\email{\href{mailto:manuel.canizares@ricam.oeaw.ac.at}{\textrm{manuel.canizares@ricam.oeaw.ac.at}}}

\author[T. Zacharopoulos]{Thanasis Zacharopoulos}
\address[T.\ Zacharopoulos]{Department of Mathematics, Aarhus University, NY Munkegade 118, 8000 Aarhus C, Denmark}
\email{\href{mailto:thanzacharop@math.au.dk}{\textrm{thanzacharop@math.au.dk}}}

\begin{document}

\begin{abstract}
The initial-to-final-state inverse problem for the Schr\"odinger equation consists in determining uniquely the Hamiltonian that generates the evolution, assuming the knowledge of the initial-to-final-state map. 
This functional maps each initial state $f\in L^2(\R^n)$ to the corresponding final state at a fixed time $T$. The problem was formulated by Caro and Ruiz in the case of Hamiltonians arising from time-dependent bounded electric potentials that exhibit super-exponential decay at infinity. Caro, Parissis and the authors of this article established that uniqueness also holds for time-independent bounded potentials with super-linear decay at infinity. 

In this paper, we quantify the above uniqueness results establishing that the potentials are stable under small changes of the initial-to-final-state maps. In the case of time-dependent potentials we get a stability of logarithmic type. A notable improvement is achieved when the potentials are time-independent, where under this assumption we prove H\"older stability estimates.
\end{abstract}

\date{\today}

\subjclass[2020]{Primary 35R30; Secondary 35J10, 81U40.}
\keywords{Initial-to-final-state map, Schr\"odinger equation, uniqueness, stability, distance, inverse problems}

\maketitle

%%%%%%%%%% new section %%%%%%%%%% new section %%%%%%%%%% new section

\section{Introduction}

We consider the initial-value problem for the Schr\"odinger equation with an electric potential
\begin{equation}\label{eq: IVP_Schrodinger}
	\left\{
		\begin{aligned}
		& i\partial_\tm u = - \Delta u + V u & & \textnormal{in} \enspace (0,T) \times \R^n, \\
		& u(0, \centerdot) = f &  & \textnormal{in} \enspace \R^n.
		\end{aligned}
	\right.
\end{equation}
It is a well known fact that, if $V \in L^1((0,T); L^\infty(\R^n))$, then for every $f \in L^2(\R^n)$
there exists a unique solution 
$u \in C([0,T]; L^2(\R^n))$ to \eqref{eq: IVP_Schrodinger}. This can be seen as a particular case of a much stronger result proved in \cite{zbMATH02204588}.
Moreover, the linear map 
$ f \in L^2(\R^n) \mapsto u \in C([0, T]; L^2(\R^n)) $ is bounded.
Such solutions will be referred to as \emph{physical solutions}.

Caro and Ruiz formulated in \cite{zbMATH07801151} an
inverse problem consisting in determining the electric potential $V:[0,T]\times\R^n\to\C$ from data 
measured only at the initial and final times. 
These data were quantified by the \textit{initial-to-final-state map} operator, defined over the physical solutions $u$ of the problem \eqref{eq: IVP_Schrodinger}, as
\[
\mathcal{U}_T : f \in L^2(\R^n) \mapsto u(T, \centerdot) \in L^2(\R^n),
\]
which is a bounded operator in $L^2 (\R^n)$.
In {\cite{zbMATH07801151}}, Caro and Ruiz proved that $\mathcal{U}_T$ uniquely determines the potential $V = V(\tm, \x)$
in dimension $n \geq 2$ 
whenever $V \in L^1((0, T); L^\infty (\R^n))$ has super-exponential decay; that is,
$ e^{\rho |\x|} V \in L^\infty((0, T) \times \R^n) $ for all $\rho > 0$.
The necessity of the super-exponential decay condition arises from the use of a family of solutions, usually called complex 
geometrical optics (CGO) solutions, that grow exponentially at infinity. 
In \cite{arXiv:2512.04796}, they extended this uniqueness result to unbounded time-dependent potentials with local $L^q$-type singularities, but still under the super-exponential decay at infinity assumption. 

Under the assumption of a time-independent and bounded potential with super-linear decay at infinity, Caro, Parissis and the authors of this article 
recovered in \cite{zbMATH08122191} the above uniqueness result; see subsection \ref{subsec:potential_decay} for the definition of super-linear decay. 
The time-independence of the potential allowed us to completely avoid CGO solutions by constructing suitable time-harmonic solutions. 
More recently, this uniqueness result was extended in \cite{arXiv:2602.12122} to unbounded time-independent potentials that may exhibit $L^q$-type singularities and requiring only an $L^1$-type decay at infinity.

Building on the works \cite{zbMATH08122191} and \cite{zbMATH07801151}, a natural direction is to quantify these uniqueness results, and investigate how stable could a potential reconstruction be, under controlled errors on the initial-to-final-state measurements. The usual starting point for analytical inverse problems in PDEs consists in relating the measurement data to the unknown quantity via an identity of, in this case, the form
\begin{equation}\label{eq: identity}
	i \int_{\R^n} (\mathcal{U}_T^1 - \mathcal{U}_T^2) f \, \overline{g} \; = \int_\Sigma (V_1 - V_2)u_1 \overline{u_2}\,,
\end{equation}
for $u_1$ and $u_2$ physical solutions of \eqref{eq: IVP_Schrodinger} associated with initial data $f$ and final data $g$, respectively. This identity was proven in \cite[Proposition 4.1]{zbMATH07801151}. Above and from now on we denote the unbounded strip $\Sigma \defeq (0, T) \times \R^n$ with $n\geq 2$.

Whenever the data for two potentials are identical, the LHS will be zero, and thus so is the RHS. Probing this right hand side with a special choice of solutions allows to recover the Fourier transform of $V_1-V_2$, and show uniqueness by the Fourier inversion theorem.

To obtain stability estimates, one may use the same identity to estimate the size of this Fourier transform in terms of the proximity between $\mathcal U_T^1$ and $\mathcal U_T^2$. The Fourier transform of $V_1-V_2$ in the RHS of \eqref{eq: identity} can be extracted by plugging in the same kind of special solutions. Then we have to estimate its size in the LHS by the operator norm of the measurements, controlling the growth of these solutions.

% The decay of the potential allows us to extend the RHS for these special solutions. However, it is not clear how to make sense of the LHS for this larger class of solutions, since the initial-to-final state map would in principle become unbounded for initial data that does not belong to $L^2(\R^n)$.

%
Ultimately, we would be interested in an estimate of the form 
\begin{equation}\label{eq: stab_est_nat}
\|V_1-V_2\|_{\X} \leq \upomega \big( \|\cal{U}_T^1 - \cal{U}_T^2 \|_{\mathcal Y}\big),
\end{equation}
where $\upomega: [0, \infty) \to [0, \infty)$ is a modulus of continuity and $\X$ and $\mathcal Y$ are Banach spaces to be specified. 
%See subsection \ref{subsec:def} for the definition of $\cal{L}(L^2(\R^n))$.
 
Even though $\| \cdot \|_{\cal Y} = \| \cdot \|_{\cal{L}(L^2(\R^n))} $ would be the natural operator norm for the map $\mathcal{U}_T$, proving a stability estimate like \eqref{eq: stab_est_nat} seems to be highly non-trivial. 
The main obstacle is that the aforementioned solutions that are used in \cite{zbMATH08122191} and \cite{zbMATH07801151} to obtain uniqueness are not physical; they present a growth at infinity in the space variable. This complicates the arising of the norm $\|\cal{U}_T^1 - \cal{U}_T^2 \|_{\cal{L}(L^2(\R^n))}$ in the final estimates. Although the decay of the potential allows us to extend the RHS of \eqref{eq: identity} for these special solutions, it is not clear how to make sense of the LHS for this larger class of solutions, since the initial-to-final state map would in principle become unbounded for initial data that do not belong to $L^2(\R^n)$.   
%In the former case they are stationary states that present a growth of polynomial type at infinity, while in the latter case CGO solutions are used which can grow exponentially in some directions of $\R^n$.  

To deal with this problem we introduce a quantity, $\dist_w(\cdot, \cdot): \cal{L}(L^2(\R^n))\times\cal{L}(L^2(\R^n)) \to\R_+$, as 
\[
\dist_w(\cal{U}_T^1, \cal{U}_T^2) \defeq
\sup_{\substack{u_1 \in \cal{P}_{V_1}\\ u_2 \in \cal{P}_{\overline{V}_2}}} \frac{\Big| \displaystyle \int_\Sigma (\cal{U}_T^1-\cal{U}_T^2) u_1(0, \cd) \ovl{u_2(T, \cd)}\Big|}{\|u_1\|_{L^2_w(\Sigma)} \| u_2\|_{L^2_w(\Sigma)}},
\]
where $\|\centerdot\|_{L^2_w(\Sigma)}$ are suitable weighted $L^2$ spaces in which the growth of the necessary special solutions can be estimated, see section \ref{subsec: spaces} for details. Here $\cal{P}_V$, see definition \eqref{eq: space_physical_solutions}, denotes the space of physical solutions corresponding to the potential $V$.
We prove in Proposition \ref{prop: distance} that this distance becomes a quasi-metric on a subset of $\cal{L}(L^2(\R^n))$, under an a priori assumption for the potentials.
%see \ref{subsec:def} for the definitions of distance and quasi-metric. 
For real potentials and in the case of $w\equiv 1$ (i.e. for $L^2_w = L^2$), it holds that $\dist_w(\cal{U}_T^1, \cal{U}_T^2)=\|\cal{U}_T^1 - \cal{U}_T^2\|_{\cal{L}(L^2(\R^n))}$. In general, for weights $0<w(x)\leq1$, we have that
\[
\|\cal{U}_T^1 - \cal{U}_T^2\|_{\cal{L}(L^2(\R^n))} \lec \dist_w(\cal{U}_T^1, \cal{U}_T^2).
\]
In the sequel, we may denote the distance as $\dist_\delta$ or $\dist_{\delta,\ve}$, whenever the weight $w$ has a dependence in either $\delta$ or $\delta$ and $\ve$, to avoid the cluttering of notation.

We want this distance to be bounded, which is guaranteed by requiring slightly stronger decay assumptions for the potentials than the ones in \cite{zbMATH08122191} and \cite{zbMATH07801151}. 
To be more specific, for time-independent potentials $V\in L^1(\R^n)\cap L^\infty(\R^n)$ and for any $\delta>1/2$, we assume the following decay 
\begin{equation}\label{eq: potential_decay}
\| V \|_\delta \defeq  \| \langle \cdot \rangle^{2\delta} V\|_{L^\infty(\R^n)} < \infty,
\end{equation}
where we denote $\langle x \rangle \defeq (1+|x|^2)^{1/2}$. 
 %

%%%%%% 
For the same reason, in the case of time-dependent $V\in L^1((0, T); L^\infty(\R^n))$, we suppose that
\begin{equation}\label{eq: super_exponential}
		\| V \|_{\delta,\ve} \defeq \|e^{2|\cdot|^{1+\ve}}\langle \cdot \rangle^{2\delta}V \|_{L^\infty((0,T)\times \R^n)} <\infty,
\end{equation}
for any $\delta>1/2$ and $\ve>0$, which at infinity implies the super-exponential decay.
Furthermore, in the case of time-independent potentials, an extra integrability condition naturally arises through the estimates in the proof of the forthcoming Theorem \ref{th: stability_t_independent}. That is, 
\begin{equation}\label{eq: int_cond_V}
\int_{\R^n} |x| |V(x)| \, \dd x < \infty. 
\end{equation}

We now present the main theorems of this article. Denote with $\Im V$ the imaginary part of $V$. In the case of time-independent potentials, as A. Ruiz conjectured, we get a H\"older modulus of continuity: 
%%%%%%%THEOREM%%%%%%%%THEOREM%%%%%%%%%THEOREM%%%%%%%%
\begin{theorem}\label{th: stability_t_independent}
Let $0<T<\infty$. Consider two time-independent potentials $V_1, V_2 \in L^1(\R^n) \cap L^\infty(\R^n)$ with $n\geq 2$, $\| V_j \|_\delta< \infty$ and $\| \Im V_j \|_\delta \leq \frac{1}{4T}$ for $j\in\{1, 2\}$, satisfying the a priori assumption $\| V_1 - V_2\|_\delta \leq R_0$ for some fixed constant $0<R_0<\infty$. 
Assume also that $V_1$ and $V_2$ meet the integrability condition \eqref{eq: int_cond_V}.

Let $\cal{U}_T^1$, $\cal{U}_T^2$ denote the corresponding initial-to-final-state maps. Then, providing that $ \dist_{\delta} (\cal{U}_T^1, \cal{U}_T^2)$ is sufficiently small, for any $\delta>1/2$ there holds
\begin{equation}\label{eq: stability1_est_t_independent}
\|V_1 - V_2\|_{H^{-1}(\R^n)} \lec \dist_{\delta} (\cal{U}_T^1, \cal{U}_T^2)^{\frac{2}{n(n+2)}}.
\end{equation}
Additionally, for $\delta>n/2$, the following improved stability estimate holds
\begin{equation}\label{eq: stability2_est_t_independent}
\|V_1 - V_2\|_{H^{-1}(\R^n)} \lec \dist_{\delta} (\cal{U}_T^1, \cal{U}_T^2)^{\frac{1}{n}}. 
\end{equation}
\end{theorem}
%%%%%%%%%%%%

In the case of time-dependent potentials, we obtain a stability result of logarithmic type:

%%%%%%%THEOREM%%%%%%%%THEOREM%%%%%%%%%THEOREM%%%%%%%%
\begin{theorem}\label{th: stability_t_dependent}
	Let $0<T<\infty$ and $n\geq 2$. Consider two time-dependent potentials $V_1, V_2 \in L^1((0,T);L^\infty(\R^n))$ with decay of the form $\|V_j\|_{\delta, \ve}< \infty$ and $\|\Im V_j\|_{\delta, \ve}\leq \frac{1}{4T} $ for $j\in\{1, 2\}$. Assume also that the potentials satisfy the a priori assumption $\|V_1 - V_2\|_{\delta, \ve} \leq R_0$ for some fixed constant $0<R_0<\infty$.
	 
	Let $\cal{U}_T^1$, $\cal{U}_T^2$ denote the corresponding initial-to-final-state maps. Then, for any $\delta >1/2$ and $ \varepsilon>0$, there holds 
	\begin{equation}\label{eq: stability1_est_t_dependent}
	\|V_1 - V_2\|_{H^{-1}(\R^n)} \lesssim \omega_\ve(\dist_{\delta,\varepsilon} (\cal{U}_T^1, \cal{U}_T^2)),
	\end{equation}
	under the condition that $\dist_{\delta,\varepsilon} (\cal{U}_T^1, \cal{U}_T^2)$ is sufficiently small. Here $\omega_\ve$ is a modulus of continuity satisfying that
	\[
	\omega_\ve(t)\leq\left(\frac{|\log(t)|}{5\varepsilon}\right)^{\frac{-\varepsilon}{n(n+3)(1+\varepsilon)}},\quad 0<t<1/e.
	\]
\end{theorem}

%%%%%%REMARK%%%%%%%%REMARK%%%%%%%%REMARK%%%%%%%

%
\begin{remark}\label{remark1}
	We can observe, that in the case of Theorem \ref{th: stability_t_dependent}, the distance gets stronger as $\varepsilon\to 0$ but, in turn, the exponent of the stability estimate becomes worse, with the limit estimate being \[\|V_1 - V_2\|_{H^{-1}(\R^n)} \lec 1.\] 
\end{remark}

%%%%%%%%%%%%%%%%%%%%%%%%%%
%
 The logarithmic stability illustrates the ill-posedness of the inverse problem, which lays in accordance with classical estimates in inverse problems for diverse partial differential equations. This list of such estimates is extensive, see for instance \cite{Alessandrini01011988, Stefanov1990, Caro2011,Caro2013}. However, a remarkable fact is that considering time-independent potentials results in a significant improvement with respect to the modulus of continuity. 
 
 This comes as a consequence of the fact that the solutions we can use present a weaker growth at infinity. In particular, time-dependent potentials ask for the use of CGO solutions, which grow exponentially at infinity. In the case of time-independent potentials, recall that it is possible to use time-harmonic solutions which only present a polynomial growth. There exist other instances in which stronger assumptions on the recovered parameters allow for better stability. One example is that of Lipschitz stability for piecewise constant conductivities in the Calder\'on problem \cite{Alessandrini2005}.

 Although there exist multiple references concerning stability estimates for the dynamical Schr\"odinger equation \cite{Baudouin2002,BenAcha2017, Bellassoued2017,Lai2024,Bellassoued2010}, they tend to focus on using boundary values, e.g. the Dirichlet-to-Neumann map, as measurement data. To the best of our knowledge, this is the first time in the literature that stability estimates analogous to the ones of the Theorems \ref{th: stability_t_independent} and \ref{th: stability_t_dependent} are achieved by means the initial-to-final-state map operator $\cal{U}_T$.

A remark in this direction is that of Kian and Soccorsi, where in \cite{Kian2019} they proved a stability estimate of H\"older type for time-dependent electromagnetic potentials in the setting of the Schr\"odinger equation. We would like to emphasize that, in this instance, the measurements are performed in the boundary of a bounded domain at all times $t\in (0,T)$, while our setting considers only measurements at times $t=0$ and $t=T$.

%%%%%%%%%%%%%%%%%%%%
%%%%%%%%

%%%%%%%%%%%%%%
\vspace{0.2cm}

The distance $\dist_w(\cdot, \cdot): \cal{L}(L^2(\R^n))\times\cal{L}(L^2(\R^n)) \to \R_+$ is the main novelty of this article and it effectively allows us to handle the previously discussed barriers. 
An important ingredient of the paper is a Runge type approximation argument. We will prove that physical solutions can approximate CGO and time-harmonic in the suitable weighted $L^2$-topologies; see Lemmas \ref{lem: density_of_solutions_time} and \ref{lem: density_of_solutions_t_independent}. These density results are also the main reason that we choose to measure the physical solutions with respect to the same weighted $L^2$-topologies in the definition of $\dist_w(\cal{U}_T^1, \cal{U}_T^2)$. 

We are aware that this distance comes with its own conceptual limitations. For instance, the denominator includes an $L^2$ norm that integrates the solutions over the whole interval which obscures the actual distance measuring between two initial-to-final-state maps. However, since this distance is always bounded, it enables it to be used as a quantifier for our estimates provided that it is sufficiently small. 
It would be desirable to change this quantity for a norm only at initial and final times. Nevertheless, the dispersive nature of the equation is an obstacle yet to overcome. A direction of future work is to refine the estimate by finding a stronger quantifier for the initial-to-final maps.
Finally, we remark that the modulus of continuity is optimal for the methods applied, since it is dictated by the growth of the special solutions in conjunction with the distance. It is an open question whether better exponents can be achieved through other strategies.

%%%%%%%

The structure of the rest of the paper goes as follows. Function spaces that are used throughout this article are defined in Section \ref{sec: preliminaries}. In Section \ref{sec: distance_and_density} we define the distance and prove that it is a quasi-metric under an a priori assumption for the potentials. Section \ref{sec: solutions_construction} is dedicated to recalling the construction of the stationary states and the CGO solutions. In Section \ref{sec: density} we prove that physical solutions of the Schr\"odinger equation approximate both the time-harmonic and CGO solutions in some specific weighted  $L^2$-topologies. As a consequence we use this fact in order to surface the distance as a quantifier in the final estimates. Finally, we provide the proofs of the Theorems \ref{th: stability_t_independent} and \ref{th: stability_t_dependent} in Section \ref{sec: stability_estimates}.

%%%%%%%%%%%SECTION%%%%%%%%%%SECTION%%%%%%%%%SECTION%%%%%%%%%%%
\section{Preliminaries}\label{sec: preliminaries} 
In this section we introduce some functional spaces and notions that will be necessary for our approach and used throughout this article.  

%%%%%%%SUBSECTION%%%%%%SUBSECTION%%%%%%%%%%%%%%%
\subsection{Definitions}\label{subsec:def}
In general, for a non-empty set $D$, a real-valued function $d(x, y)$ defined for $x, y\in D$ is called a \textit{distance function} if $d(x, y) \geq 0$ for all $x, y\in D$, $d(x, y)=0$ if and only if $x=y$ and $d(x, y) = d(y, x)$ for all $x, y \in D$. If moreover, there exists a real number $C\geq 1$ such that 
\[
d(x, y) \leq C(d(x, z) + d(z, y))
\]
for all $x, y, z \in D$, then the distance function $d(\cdot, \cdot)$ is said to be a quasi-metric on $D$, with constant $C$. In the special case that $C=1$, the distance function becomes a metric on $D$.  \\
%Here in order to prove that $\dist(\cdot, \cdot): \cal{L}(L^2(\R^n)) \to\R_+$ satisfies a quasi-triangle inequality, an a priori assumption, see assumption \ref{as: apriori_as}, is necessary. 

We denote with $\cal{L}(L^2(\R^n)) \equiv \cal{L}(L^2(\R^n), L^2(\R^n))$ the space of continuous linear operators from $L^2(\R^n)$ to $L^2(\R^n)$ equipped with the norm 
\[
\|U\|_{\cal{L}(L^2(\R^n))} = \sup_{  0 \neq f\in L^2(\R^n) } \frac{ \| Uf \|_{L^2(\R^n)}}{ \|f\|_{L^2(\R^n)}}
= \sup_{0\neq f, g\in L^2(\R^n)} \frac{|\langle Uf, g\rangle|}{\|f\|_{L^2(\R^n)}\|g\|_{L^2(\R^n)}}.
\]
The second equality is due to duality, where $\langle \cdot, \cdot\rangle$ represents the usual inner product in $L^2(\R^n)$.

%%%%%%SUBSECTION%%%%%%%%SUBSECTION%%%%%%%%%
\subsection{Weaker forms of decay for the potentials}\label{subsec:potential_decay}

Following the convention in \cite{zbMATH08122191}, we say that a time-independent potential $V\in L^\infty(\R^n)$ has super-linear decay at infinity if the following condition holds:
\begin{equation}\label{eq: pol_decay_V}
\vvvert V \vvvert = \sum_{j\in\mathbb N_0} 2^j \|V\|_{L^\infty(D_j)} < \infty,
\end{equation}
where
\begin{equation}\label{eq: disks}
D_0 := \{ x\in \R^n : |x|\leq 1\}, \quad D_j := \{ x\in\R^n : 2^{j-1} \leq |x| \leq 2^j \}, \,\, \forall j\in\mathbb N .
\end{equation}
This is the only decay needed to construct the stationary-states in subsection \ref{sec: stationary states}.

%%%%%%%%%%%%%%%%%%%%%
Notice that  $\vvvert V \vvvert \lec_\delta \| V \|_\delta$. Indeed, by the definition of $D_j$ we have that 
$2^{2\delta j} \leq 2^{2\delta} (1+|x|^2)^\delta$ for every $j\geq 0$. So, for $x\in D_j$ we have that 
$2^{2\delta j}|V(x)| \leq 2^{2\delta}|\langle x\rangle^{2\delta} V(x)| \leq 2^{2\delta}\|\langle \cdot \rangle^{2\delta} V\|_{L^\infty(\R^n)}$ for any $j\in\mathbb N_0$. As a result we get 
\[
\sum_{j\in\N_0} 2^j \|V\|_{L^\infty(D_j)} = \sum_{j\in\N_0} 2^{-j(2\delta-1)} 2^{2\delta j} \|V\|_{L^\infty(D_j)} 
\leq 2^{2\delta}\|\langle \cdot \rangle^{2\delta} V\|_{L^\infty(\R^n)} \sum_{j\in\N_0} 2^{-j (2\delta-1)}
\lec_\delta \| \langle \cdot \rangle^{2\delta} V\|_{L^\infty(\R^n)},
\]
for any $\delta > 1/2$.

%%%%%%%%%%%

In the time-dependent case $V\in L^\infty((0, T)\times \R^n)$ must fulfill the following decay condition for the construction of CGO solutions:
\begin{equation}\label{eq: polynomial_decay_V_t}
\sum_{k=0}^\infty 2^{k(1/2+\theta)}\|V\|_{L^\infty(\Sigma^\nu_k)}<\infty,
\end{equation}
for some $\theta\in(0,1/2)$. Here, for $\nu\in\R^n\setminus \{0\}$, the sequence of strips $\{ \Sigma_j : j\in \{0\}\cup \N\} \subset \Sigma$ in the direction $\nu$ is denoted by
\begin{equation}\label{eq: strips}
\Sigma^\nu_k = (0,T) \times \Omega^\nu_k,
\end{equation}
where 
\begin{equation}\label{eq: d_strips_space}
\Omega^\nu_0 = \{ x \in \R^n : |x \cdot \nu| \leq |\nu|  \}, \quad
\Omega^\nu_k = \{ x \in \R^n : 2^{k - 1} |\nu|< |x \cdot \nu| \leq 2^k |\nu|\}  \enspace \forall k \in \N, 
\end{equation}
represent dyadic strips in space. This decay assumption is enough to guarantee the construction of the CGO solutions in subsection \ref{sect: CGO} \cite{zbMATH07801151}.

%%%%%%%%%%%%

%%%%%%%%%SUBSECTION%%%%%%%%%%SUBSECTION%%%%%%%%%
\subsection{Function spaces}\label{subsec: spaces} For a compact interval $I\subset \R$ and a Banach space $X$, we consider the Bochner-Lebesgue spaces $L^r(I; X)$ for $r\in[1,+\infty]$ consisting of functions $u: t\in I \mapsto u(t, \cd) \in X $, equipped with the norm 
\begin{equation}\label{eq: mixed_norms}
\|u\|_{L^r(I; X)}:= \Big( \int_I \|u(t, \cd)\|_X^r\, \dd t \Big)^{1/r}.
\end{equation}
We also consider the Bochner spaces $C(I; X)$ of continuous functions $u: t\in I \mapsto u(t, \cd) \in X $ with norm
\[
\|u\|_{C(I; X)}:= \sup_{t\in I}\|u(t, \cd)\|_{X},
\]
and we use the convention $\|u\|_{C(I; X)}=\infty$ whenever the function $t\mapsto \|u(t, \cd)\|_X$ is not continuous.
%%%%%%
In the most of the cases we will consider $r=2$, $I$ will represent a finite time interval $[0, T]$ and the Banach space $X$ above will either be $L^2(\R^n)$ or a weighted $L^2$-space.
%

%%%%%%WEIGHTED%%%%%%%
Throughout this paper, we will make use of weighted $L^2$ spaces. In general, see for example \cite{Rudin}, a positive weight $w:\R^n \mapsto(0,+\infty)$ defines a Borel measure $\dd \mu = w(x)\dd t \, \dd x$ and thus a Hilbert space $L^2_w(\Sigma)$ via the inner product 
\begin{equation}\label{eq: inner_product_weighted}
\langle u, v\rangle_w = \int_{(0,T)\times \R^n}u(t,x)\overline{v(t,x)}w(x)\,\dd x \, \dd t,
\end{equation}
and consequently via the norm
\[
\|u\|^2_{L^2_w(\Sigma)}=  \int_{(0, T)\times \R^n}|u(t,x)|^2 w(x)\,\dd x \, \dd t.
\]
Equivalently, a function $u$ belongs to $L^2_w(\Sigma)$ if and only if $w^{1/2}\, u$ belongs to $L^2(\Sigma)$.
By Riesz's representation theorem, the dual space $(L^2_w(\Sigma))^*$ can be identified with $L^2_{w^{-1}}(\Sigma)$. Notice that, due to Tonelli theorem, the space-time integral space $L^2_w(\Sigma)$ can be identified with the Bochner space $ L^2((0, T); L^2_w(\R^n))$. They are isometrically isomorphic and hence they possess the same Hilbert-space structure. In order to ease the notation we will mostly refer to these spaces just by $L^2_w$. Moreover, we will often drop notation by just writing $L^2_TL^2_w$ and $C_TL^2_w$ instead of $L^2((0, T); L^2_w(\R^n))$ and $C([0, T]; L^2_w(\R^n))$ respectively.

There are two specific weighted $L^2$ spaces that play a great role in this article, due to the special solutions that we construct.  Specifically, in the general case of time-dependent potentials, we consider solutions in the space $L^2_w$ with $w=w_{\delta,\ve}(x)\defeq e^{-2|x|^{1+\varepsilon}}\langle x\rangle^{-2\delta}$ for $\delta>1/2$ and $\ve>0$. In the case of time-independent potentials, we use weaker spaces defined by the weight $w_{\delta}(x)\defeq \langle x\rangle^{-2\delta}$, for $\delta>1/2$. Of course, since $w\leq 1$, physical solutions belong to both of these weighted spaces. 
Note that in the last case, $L^2_{w_\delta} = L^2((0,T);L^2_{-\delta}(\R^n))$, where $L^2_{-\delta}(\R^n)$ is the classical potential space used, for instance, in the resolvent estimates due to Agmon \cite{Agmon1975}.

 We will also make use of weighted $L^\infty$ spaces, denoted as $L^\infty_w(\Sigma)$, comprised by equivalence classes of functions $u$ such that $ u \, w \in L^\infty(\Sigma)$, with norm
 \[
 \|u\|_{L^\infty_w(\Sigma)}=\esssup_{(t,x)\in\Sigma} |w(x) u(t,x)|.
 \]
However, to avoid cluttering notation, we will use the following notation for the suitable spaces to quantify the decay of the potentials at study:
\[
\|V\|_w\defeq \|V\|_{L^\infty_{w^{-1}}(\Sigma)}
\]
We will repeatedly use the following inequalities that follow from H\"older's inequality:
\begin{equation}\label{eq: Holder_weighted}
\int_{\Sigma} Vu\overline{v}\leq \|V\|_w\|u\|_{L^2_w(\Sigma)}\|v\|_{L^2_w(\Sigma)},
\end{equation}
and consequently, by duality,
\[\|V u\|_{L^2_{w^{-1}}(\Sigma)}\leq \|V\|_w\|u\|_{L^2_w(\Sigma)}.\]
%%%%%%%%%%

In the particular cases where $w=w_\delta$ and $w=w_{\delta, \ve}$, we explicitly get  
\[
\|V\|_{w_\delta} = \| \langle \cdot \rangle^{2\delta} V\|_{L^\infty(\R^n)} = \|V\|_\delta
\quad \text{and } \quad
\|V\|_{w_{\delta, \ve}} = \| e^{2|\cdot|^{1+\ve}} \langle \cdot \rangle^{2\delta} V\|_{L^\infty(\Sigma)}=\|V\|_{\delta, \ve} .
\]
Recall from the introduction that this is the exact decay that is required in Theorems \ref{th: stability_t_independent} and \ref{th: stability_t_dependent}, for time-independent and time-dependent potentials respectively.

%%%%%%%%%%%%%%%%%%%%%%%%%%%%
Finally, we will briefly mention the Banach spaces that are used for the construction of the stationary states, for potentials with the super-linear decay \eqref{eq: pol_decay_V}. The Banach space $B(\R^n)$ is the set of all functions 
$f\in L^2(\R^n)$ for which the norm
\[
\|f\|_{B(\R^n)} := \sum_{j\in \mathbb N_0} 2^{j/2}\|f\|_{L^2(D_j)}  
\]
is finite.
Its dual space $B^*(\R^n)$, \cite[\S 14.1]{HormII}, consists of all $f\in L^2_{\loc}(\R^n)$ such that 
\[
\|f\|_{B^*(\R^n)} := \sup_{j\in\mathbb N_0}\big( 2^{-j/2} \| f \|_{L^2(D_j)}\big) < \infty.
\]
The disks $D_0$ and $D_j$, for $j\in \mathbb N$, are exactly like the ones defined in \eqref{eq: disks}. 
It is not hard to see that it holds 
\begin{equation}\label{eq: bstarl2delta}
	\|f\|_{L^2_{-\delta}(\R^n)} \lec \|f\|_{B^*(\R^n)} \leq \|f\|_{L^2(\R^n)} \leq \|f\|_{B(\R^n)} \lec \|f\|_{L^2_\delta(\R^n)}
\end{equation}
and so, for any $\delta >1/2$, we have the following chain of inclusions
\[
 L^2_{\delta}(\R^n) \subset B(\R^n) \subset L^2(\R^n) \subset B^*(\R^n) \subset L^2_{-\delta}(\R^n).
\]
%
%%%%%%%

%%%%%%%SUBSECTION%%%%%%%%SUBSECTION%%%%%%%%%
\subsection{Almost conservation of mass of physical solutions}
Let $u\in C_TL^2$ be a solution of the initial value problem \eqref{eq: IVP_Schrodinger} with $u(0,\cd)=f$. 
In case that $V\in \R$ it is well-known that the conservation of mass $\|u\|_{C_TL^2}= \|u(t, \cd)\|_{L^2}$, holds for every $t\in[0, T]$.  Complex potentials violate this conservation of mass, as the components with positive imaginary part produce amplification, and the negative imaginary part produces absorption. 

Below, we prove that this phenomena can be controlled whenever this imaginary part is bounded.
We denote with $\Re(z)$ and $\Im(z)$ the real part and imaginary part respectively of a complex number $z\in\C$.
By taking partial derivative in time and integrating by parts we have that 
\begin{align*}
\d_s \int_{\R^n} |u(s, \cd)|^2 &= 2 \Re \int_{\R^n} \d_s (u(s, \cd)) \ovl{u(s, \cd)}
=2\Re \Big( i \int_{\R^n} (\Delta - V)u(s, \cd) \cdot \ovl{u(s, \cd)} \Big) \\
&= 2\Re \Big( i \int_{\R^n} |\nabla u(s, \cd)|^2 \Big) - 2\Re \Big( i \int_{\R^n} V u(s, \cd) \ovl{u(s, \cd)} \Big) 
= 2 \int_{\R^n} \Im(V(s, \cd)) u(s, \cd) \ovl{u(s, \cd)} .
\end{align*}
Hence
\[
\big| \d_s \big( \|u(s, \cd)\|_{L^2(\R^n)}^2 \big) \big| \leq 2 \|\Im V(s,\centerdot)\|_{L^\infty(\R^n)} \|u(s, \cd)\|_{L^2(\R^n)}^2, 
\]
or equivalently 
\begin{equation}\label{eq: partial_s_log_physical_bounds}
-2\|\Im V(s,\centerdot)\|_{L^\infty(\R^n)}  \leq \d_s \Big( \ln \big(\|u(s, \cd)\|_{L^2(\R^n)}^2 \big) \Big) \leq 2 \|\Im V(s,\centerdot)\|_{L^\infty(\R^n)}.
\end{equation}
Integrating the above inequalities over $[0, t]$ for $t\in(0, T]$, we infer that 
%
%\[
%e^{- \int_0^t\|V(s,\centerdot)\|_{L^\infty(\R^n))}\dd s} \| u(0, \cd)\|_{L^2(\R^n)} \leq \|u(t, \cd)\|_{L^2(\R^n)} 
%\leq e^{ \int_0^t \|V(s,\centerdot)\|_{L^\infty(\R^n))}\dd s} \|u(0, \cd)\|_{L^2(\R^n)}, 
%\]
%
%which implies  
%
\begin{equation}\label{eq: quasi-conservation2}
e^{- \|\Im V\|_{L^1((0,T);L^\infty(\R^n))}}  \| u(0, \cd)\|_{L^2(\R^n)} \leq \|u(t, \cd)\|_{L^2(\R^n)} 
\leq e^{\|\Im V\|_{L^1((0,T);L^\infty(\R^n))}} \|u(0, \cd)\|_{L^2(\R^n)}
\end{equation}
for every $t\in[0, T]$.

Moreover, integrating \eqref{eq: partial_s_log_physical_bounds} over $[t, T]$ for $t\in[0, T)$ we can also conclude that 
\begin{equation}\label{eq: quasi-conservation3}
e^{- \|\Im V\|_{L^1((0,T);L^\infty(\R^n))}} \| u (t, \cd)\|_{L^2(\R^n)} \leq \|u(T, \cd)\|_{L^2(\R^n)} \leq e^{ \|\Im V\|_{L^1((0,T);L^\infty(\R^n))}} \| u(t, \cd)\|_{L^2(\R^n)}.
\end{equation}
Obviously, in the case of time-independent complex potentials the same estimates hold with $ T\|\Im V\|_{L^\infty(\R^n)}=\|\Im V\|_{L^1((0,T);L^\infty(\R^n))} $. 

%%%%%%%%%%%

%%%%%%%%%%SECTION%%%%%%%%%SECTION%%%%%%%%%%SECTION%%%%%%%
\section{Distance of the initial-to-final-state maps}\label{sec: distance_and_density}
The main objective here is to define an ad hoc quantity 
that will allow us to quantify all the necessary information for the stability estimates.  
The goal is to present the distance and related results under the most general assumptions possible. For this reason, the analysis below is performed for time-dependent potentials and a weighted $L^2_w$-space with weight $w:\Sigma\to[0,\infty)$ satisfying $0<w(t,x)\leq1$ for all $(t,x)\in\Sigma$. 

The notion of general weights allows us to provide combined proofs of technical lemmas that work the same way for both cases of time-dependent and time-independent potentials.

For some finite time $T$, we define the space of physical solutions to the Schr\"odinger equation with potential $V$ as follows:
\begin{equation}\label{eq: space_physical_solutions}
\cal{P}_V \defeq \{ u\in C([0, T]; L^2(\R^n)) : (i\d_t + \Delta - V)u=0\}.
\end{equation}
Whenever there is no ambiguity on the potential, we will drop the notation to $\cal P$ instead.

Recall that under the restriction $0<w\leq1$, the physical solutions of the Schr\"odinger equation \eqref{eq: IVP_Schrodinger} obviously belong to $L^2_w$.
We define the quantity 
%
%%%%%%DEFINITION_DISTANCE%%%%%%%%%%%%%%%%
\begin{equation}\label{eq: distance}
\dist_w(\cal{U}_T^1, \cal{U}_T^2) \defeq
\sup_{\substack{u_1 \in \cal{P}_{V_1}\\ u_2 \in \cal{P}_{\overline{V}_2}}} \frac{\Big| \displaystyle \int_\Sigma (\cal{U}_T^1-\cal{U}_T^2) u_1(0, \cd) \ovl{u_2(T, \cd)}\Big|}{\|u_1\|_{L^2_w(\Sigma)} \| u_2\|_{L^2_w(\Sigma)}}
=\sup_{\substack{u_1 \in \cal{P}_{V_1}\\ u_2 \in \cal{P}_{\overline{V}_2}}} \frac{\Big| \displaystyle \int_\Sigma (V_1 - V_2) u_1 \ovl{u_2} \Big|}{\|u_1\|_{L^2_w(\Sigma)} \| u_2\|_{L^2_w(\Sigma)}},
\end{equation}
where the second equality above holds by the integral identity for physical solutions \eqref{eq: identity}. 

%%%%%%%remark%%%%%%%%%remark%%%%%%%%%
In the upcoming sections, we will focus in the particular cases where $w=w_\delta(x)=\langle x\rangle^{-2\delta}$ and $w=w_{\delta,\ve}(x)=e^{-2|x|^{1+\varepsilon}}\langle x\rangle^{-2\delta}$ and we will just write $\dist_\delta(\cdot, \cdot)$ and $\dist_{\delta, \ve}(\cdot, \cdot)$.
In the first case the distance depends only on $\delta$, and it will be used to obtain stability estimates in the case of time-independent potentials.
In the latter case, $\dist_{\delta, \ve}(\cdot, \cdot)$ will be used when investigating the stability for time-dependent potentials.

Observe that the quantity \eqref{eq: distance} is bounded. Indeed, by H\"older's inequality \eqref{eq: Holder_weighted}, we have that 
\begin{equation}\label{eq: dist_bound}
\dist_w(\cal{U}_T^1, \cal{U}_T^2) \leq \| V_1-V_2 \|_w.
\end{equation}

Furthermore, it is easy to see that $\|u\|_{L^2_w(\Sigma)}\leq T^{1/2} \|u\|_{C([0, T]; L^2(\R^n))}$. Using this along with the quasi-conservation of mass \eqref{eq: quasi-conservation2} and \eqref{eq: quasi-conservation3} we infer that 
\begin{align*}
\dist_{w}(\cal{U}_T^1, \cal{U}_T^2) &\geq 
 T^{-1} \sup_{\substack{u_1\in\mathcal P_{V_1}\\ u_2\in\mathcal P_{\ovl V_2}}} \frac{\Big| \displaystyle \int_{\R^n} (\cal{U}_T^1-\cal{U}_T^2) u_1(0, \cd) \ovl{u_2(T, \cd)} \Big|}{\|u_1\|_{C_TL^2} \|u_2\|_{C_TL^2}} \\
&\geq T^{-1} e^{ - \big(\|\Im V_1\|_{L^1_TL^\infty} + \|\Im V_2\|_{L^1_TL^\infty} \big)} \sup_{\substack{u_1\in\mathcal P_{V_1}\\ u_2\in\mathcal P_{\ovl V_2}}} \frac{\Big| \displaystyle \int_{\R^n} (\cal{U}_T^1-\cal{U}_T^2) u_1(0, \cd) \ovl{u_2(T, \cd)} \Big|}{\|u_1(0, \cd)\|_{L^2(\R^n)} \|u_2(T, \cd)\|_{L^2(\R^n)}},
\end{align*} 
Notice that in the case of real-valued potentials, we would just have that
\[ 
\sup_{\substack{u_1\in\mathcal P_{V_1}\\ u_2\in\mathcal P_{\ovl V_2}}} \frac{\Big| \displaystyle \int_{\R^n} (\cal{U}_T^1-\cal{U}_T^2) u_1(0, \cd) \ovl{u_2(T, \cd)} \Big|}{\|u_1\|_{C_TL^2} \| u_2\|_{ C_TL^2}} = \sup_{\substack{u_1\in\mathcal P_{V_1}\\ u_2\in\mathcal P_{\ovl V_2}}} \frac{\Big| \displaystyle \int_{\R^n} (\cal{U}_T^1-\cal{U}_T^2) u_1(0, \cd) \ovl{u_2(T, \cd)} \Big|}{\|u_1(0, \cd)\|_{L^2(\R^n)} \|u_2(T, \cd)\|_{L^2(\R^n)}},
\]
by the usual conservation of mass for solutions of Schr\"odinger equation. 

Since, by definition, 
\[ 
\|\cal{U}_T^1 - \cal{U}^2_T\|_{\cal{L}(L^2(\R^n))}=\sup_{\substack{u_1\in\mathcal P_{V_1}\\ u_2\in\mathcal P_{\ovl V_2}}} \frac{\Big| \displaystyle  \int_{\R^n} (\cal{U}_T^1-\cal{U}_T^2) u_1(0, \cd) \ovl{u_2(T, \cd)} \Big|}{\|u_1(0, \cd)\|_{L^2(\R^n)} \|u_2(T, \cd)\|_{L^2(\R^n)}},
\] 
the above estimate implies 
\begin{equation}\label{eq: bound_IF_map_distance}
	\|\cal{U}_T^1 - \cal{U}^2_T\|_{\cal{L}(L^2(\R^n))} \lec \dist_{w}(\cal{U}_T^1, \cal{U}_T^2). 
\end{equation}
%  %

%%%%%%
Subsequently, we will prove that this quantity defines a quasi-metric in the space of initial-to-final-state maps, for potentials in an admissible class that satisfy an a priori assumption. 
%%%%%%
%
In this direction, we need to prove the following lemma, which allows us to control the difference between the evolution of states with respect to different potentials.

The intuitive interpretation is that any two solutions of the Schr\"odinger equation with different potentials but same initial data stay ``close" to each other whenever the corresponding potentials are themselves ``close".
%

%%%%%%%LEMMA%%%%%%%LEMMA%%%%%LEMMA%%%%%%%%%%
\begin{lemma}\label{lem: closeness_of_solutions_time}
Let $V_j\in L^1((0, T); L^\infty(\R^n))$ with $n\geq 2$ be such that $\| V_j \|_{w}< \infty$ with weight $w:\R^n \to (0,1]$ and satisfying the smallness condition $\|  \Im V_j \|_{w} \leq \frac{1}{4T} $, for $j\in\{1,2,3\}$.
For $f,g\in L^2(\R^n)$, let $u_j, \, v_j \in C([0, T]; L^2(\R^n))$ be solutions of the corresponding initial and final value problems for the Schr\"odinger equation
\[
\begin{dcases}
(i\d_t+\Delta - V_j)u_j = 0, \quad &\textup{in}\,\, \Sigma, \\
u_j(0, \cd) = f \in L^2(\R^n), &\textup{in}\,\, \R^n
\end{dcases} 
\quad \textup{and} \quad
\begin{dcases}
(i\d_t+\Delta - \ovl{V_j})v_j = 0, \quad &\textup{in}\,\, \Sigma, \\
v_j(T, \cd) = g\in L^2(\R^n), &\textup{in}\,\, \R^n.
\end{dcases} 
\]
Then, it holds that 
\begin{equation}\label{eq: solutions_diff1_time}
\|u_j - u_k \|_{L^2_w(\Sigma)} \leq 4T \| V_j - V_k \|_{w} \|u_k\|_{L^2_w(\Sigma)}
\end{equation}
and
\begin{equation}\label{eq: solutions_diff2_time}
\|v_j - v_k \|_{L^2_w(\Sigma)} \leq 4T \| V_j - V_k \|_{w} \|v_k\|_{L^2_w(\Sigma)},
\end{equation}
for every $j \neq k \in \{1, 2, 3 \}$.
\end{lemma}
%%%%%PROOF%%%%%%%PROOF%%%%%%%PROOF%%%%%%%%
\begin{proof}
	Set $\vp := u_j - u_k$ for $j\neq k\in\{1, 2, 3\}$. Then,  $\vp \in C([0,T];L^2(\R^n))$ is a solution to the following initial value problem:
%%%
%
\begin{equation}\label{eq: artificial_IVP_w_time}
\begin{dcases}
i\d_t \vp = -\Delta \vp + V_j \vp + (V_j-V_k)u_k, \quad &\textup{in}\,\, \Sigma, \\
\vp (0, \cd) = 0  &\textup{in}\,\, \R^n.
\end{dcases} 
\end{equation}
In what follows, to avoid overflow of notation, we will omit that the integrand depends on $(s,x)$, and the integration happens with respect to $\dd x$.
By taking partial derivative in time, using \eqref{eq: artificial_IVP_w_time} and integrating by parts allows us to estimate as follows 
\begin{align*}
\d_s \int_{\R^n} |\vp(s, x)|^2\dd x &= 2 \Re \int_{\R^n} \partial_s(\vp ) \ovl{\vp} 
= 2\Re \Big( i \int_{\R^n} \big[ (\Delta - V_j)\vp  - (V_j-V_k)u_k \big] \ovl{\vp} \Big) \\
&= -2 \Re \Big( i \int_{\R^n} V_j \vp \, \ovl{\vp} \Big) - 2\Re \Big( i \int_{\R^n} (V_j-V_k) u_k \ovl{\vp} \Big) \\
&=  2 \int_{\R^n} \Im(V_j) \vp \, \ovl{\vp } - 2\Re \Big( i \int_{\R^n} (V_j-V_k) u_k \ovl{\vp} \Big)  \\
&\leq 2 \int_{\R^n} \Im(V_j) \vp  \, \ovl{\vp } + 2\,\Big| \int_{\R^n} (V_j-V_k)u_k \ovl{\vp}\, \Big|
\end{align*}
Integrating over $\int_0^T \int_0^t\dd s \, \dd t$, along with the fact that $\vp(0,\cd)=0$ and H\"older inequality gives
\begin{align*}
\int_0^T \int_{\R^n} |\vp (t,x)|^2\,\dd x \, \dd t &\leq 2 \int_0^T \int_0^T\int_{\R^n} \Im(V_j) \vp  \, \ovl{\vp}\, \dd x \, \dd s \, \dd t + 
2 \int_0^T \int_0^T \Big| \int_{\R^n} (V_j-V_k)u_k \ovl{\vp}\, \Big| \, \dd x \, \dd s \, \dd t \\
&\leq 2T \| \Im V_j\|_{w}\|\vp\|_{L^2_w}^2+2T \| V_j-V_k\|_{w} \|u_k\|_{L^2_w}\|\vp\|_{L^2_w} ,
\end{align*}
Hence, since $0<w\leq 1$, we get
\[
\|\vp\|_{L^2_w}^2\leq \|\vp\|_{L^2((0, T); L^2(\R^n))}^2 \leq 2T \| \Im V_j\|_{w}\|\vp\|_{L^2_w}^2+2T \| V_j-V_k\|_{w} \|u_k\|_{L^2_w}\|\vp\|_{L^2_w},
\]
which implies
\[
\|\vp\|_{L^2_w}\left(1-2T\| \Im V_j\|_{w}\right)\leq 2T \| V_j-V_k\|_{w} \|u_k\|_{L^2_w}.
\]
By hypothesis we have that $\|  \Im V_j \|_{w} \leq \frac{1}{4T} $, or equivalently $\left(1-2T \| \Im V_j\|_{w}\right)^{-1}\leq 2$, which concludes the proof of \eqref{eq: solutions_diff1_time}. The proof of \eqref{eq: solutions_diff2_time} is omitted since it follows by similar arguments by setting $\vp=v_j-v_k$.
\end{proof}

We close this section with the following proposition. We use Lemma \ref{lem: closeness_of_solutions_time} to prove that the distance \eqref{eq: distance} is a quasi-metric in the set of initial-to-final values that are related to potentials which are close enough.

%%%%%%%%PROPOSITION%%%%%%%PROPOSITION%%%%%%%PROPOSITION%%%%%%
\begin{proposition}\label{prop: distance}
Let $n\geq 2$ and $V_1, \, V_2 \in L^1((0, T); L^\infty(\R^n))$ with $\| V_j \|_w < \infty$, for $j\in\{1, 2\}$ and weight $w:\R^n \to (0, 1)$. The quantity defined in \eqref{eq: distance} is a distance. 

In particular, fix $R_0>0$ and for $V \in L^1((0, T); L^\infty(\R^n))$ with $\| V \|_w < \infty$ and $\| \Im {V} \|_w\leq \frac{1}{4T}$ set 
\[
\cal {T}_V := \{\wt{ V} \in  L^1((0, T); L^\infty(\R^n)) : \,\, \| \wt{V} \|_w<\infty, \, \| \Im \wt{V} \|_w \leq \frac{1}{4T} \,\,
\text{and such that} \,\, 
\| V - \wt{V} \|_w \leq \frac{R_0}{2} \}.
\]
Then, if $\cal J$ is the set of the initial-to-final-state maps corresponding to potentials in $ \cal T_V$, the quantity \eqref{eq: distance} becomes a quasi-metric on the set $\cal J$.
\end{proposition}

%%%%%%%PROOF%%%%PROOF%%%%%%%%%
\begin{proof}
Obviously, $\dist_{w}(\cal{U}_T^1, \cal{U}_T^2) \geq 0$ and $\dist_{w}(\cal{U}_T^1, \cal{U}_T^2) = \dist_{w}(\cal{U}_T^2, \cal{U}_T^1)$. In order to see that 
\[
\dist_{w}(\cal{U}_T^1, \cal{U}_T^2) = 0 \Longleftrightarrow \cal{U}_T^1 = \cal{U}_T^2,
\]
first notice that if $\dist_{w}(\cal{U}_T^1, \cal{U}_T^2) = 0$, then by \eqref{eq: bound_IF_map_distance} we have that $\| \cal{U}_T^1 - \cal{U}_T^2\|_{\cal{L}(L^2(\R^n))} = 0$ which readily implies $\cal{U}_T^1=\cal{U}_T^2$. 
In the case that $\cal{U}_T^1 = \cal{U}_T^2$, then by definition we immediately have that $\dist_{w}(\cal{U}_T^1, \cal{U}_T^2) = 0$. 
Hence, \eqref{eq: distance} is a distance.
%%

%%%%%
It remains to prove that $\dist_w(\cdot, \cdot)$ satisfies a quasi-triangle inequality under the a priori assumption for the potentials. 
To that end, consider the following initial and final value problems for the Schr\"odinger equation:
\[
\begin{dcases}
(i\d_t+\Delta - V_j)u_j = 0, \quad &\textup{in}\,\, \Sigma, \\
u_j(0, \cd) = f \in L^2(\R^n), &\textup{in}\,\, \R^n
\end{dcases} 
\quad \textup{and} \quad
\begin{dcases}
(i\d_t+\Delta - \ovl{V_j})v_j = 0, \quad &\textup{in}\,\, \Sigma, \\
v_j(T, \cd) = g\in L^2(\R^n), &\textup{in}\,\, \R^n,
\end{dcases} 
\]
for $ j\in \{1, 2, 3\}$. Of course $V_j \in L^1((0, T); L^\infty(\R^n))$ for $j\in\{ 1, 2, 3 \}$. 
It holds that 
\[
i \int_{\R^n} (\cal{U}^1_T -\cal{U}^2_T) f \ovl{g} = i\int_{\R^n} (\cal{U}^1_T - \cal{U}^3_T)f\ovl{g} + i \int_{\R^n}(\cal{U}_T^3-\cal{U}^2_T)f \ovl{g}.
\]
Assume also that $\|V_3\|_w<\infty$, $\| \Im V_3 \|_w \leq \frac{1}{4T}$ and $V_1, V_2 \in \cal{T}_{V_3}$.
By taking absolute values and dividing the above equality with the quantity $\|u_1\|_{L^2_w(\Sigma)} \|v_2\|_{L^2_w(\Sigma)}$, we infer that
\begin{align*}
&\frac{\Big| \displaystyle \int_{\R^n} (\cal{U}^1_T -\cal{U}^2_T) u_1(0, \cd) \ovl{v_2(T, \cd)}\Big|}{\|u_1\|_{L^2_w(\Sigma)} \|v_2\|_{L^2_w(\Sigma)}} \\
&\leq \frac{\Big| \displaystyle \int_{\R^n} (\cal{U}^1_T -\cal{U}^3_T) u_1(0, \cd) \ovl{v_3(T, \cd)}\Big|}{\|u_1\|_{L^2_w(\Sigma)} \|v_3\|_{L^2_w(\Sigma)}} \frac{\|v_3\|_{L^2_w(\Sigma)}}{\|v_2\|_{L^2_w(\Sigma)}}
+ \frac{\Big| \displaystyle \int_{\R^n} (\cal{U}^3_T -\cal{U}^2_T) u_3(0, \cd) \ovl{v_2(T, \cd)}\Big|}{\|u_3\|_{L^2_w(\Sigma)} \|v_2\|_{L^2_w(\Sigma)}} \frac{\|u_3\|_{L^2_w(\Sigma)}}{\|u_1\|_{L^2_w(\Sigma)}}
\end{align*}
Using the estimate \eqref{eq: solutions_diff1_time} of Lemma \ref{lem: closeness_of_solutions_time} and the fact that $V_1\in\cal{T}_{V_3}$, we have that 
\[
\frac{\|u_3\|_{L^2_w(\Sigma)}}{\|u_1\|_{L^2_w(\Sigma)}} \leq \frac{\|u_3 - u_1\|_{L^2_w(\Sigma)}}{\|u_1\|_{L^2_w(\Sigma)}} + 1 
\leq 4T \| V_3 - V_1\|_w +1 
\leq  2T R_0 +1 .
\]
Similarly, using \eqref{eq: solutions_diff2_time} instead and that $V_2\in \cal{T}_{V_3}$, we also get 
\[
\frac{\|v_3\|_{L^2_w(\Sigma)}}{\|v_2\|_{L^2_w(\Sigma)}} \leq 4T \| V_3 - V_2 \|_w +1
\leq 2T R_0 +1.
\]
Setting $ M:=  2T R_0 +1 < \infty$, we obtain 
\[
\frac{\Big| \displaystyle \int_{\R^n} (\cal{U}^1_T -\cal{U}^2_T) u_1(0, \cd) \ovl{v_2(T, \cd)}\Big|}{\|u_1\|_{L^2_w(\Sigma)} \|v_2\|_{L^2_w(\Sigma)}} 
\leq M \Big( \dist_w(\cal{U}_T^1, \cal{U}_T^3) + \dist_w(\cal{U}_T^3, \cal{U}_T^2) \Big).
\]
By taking supremum over all solutions $u_1\in \cal{P}_{V_1}$, $v_2 \in \cal{P}_{\ovl{V_2}}$, this last estimate implies 
\[
\dist_w(\cal{U}_T^1, \cal{U}_T^2) \leq M \Big( \dist_w(\cal{U}_T^1, \cal{U}_T^3) + \dist_w(\cal{U}_T^3, \cal{U}_T^2) \Big)
\]
with $M > 1$, which concludes the proof that $\dist_w(\cdot, \cdot)$ is a quasi-metric on the set $\cal{J} \subset \cal{L}(L^2(\R^n))$ of the initial-to-state-maps that correspond to the potentials in $\cal{T}_{V_3}$. 
\end{proof}

%%%%%%%

%%%%%%%%%

%
%%%%%%%SECTION%%%%%%SECTION%%%%%%%%%%%%%%%%%%
%%%%%%%%%%SECTION%%%%%%%%%SECTION%%%%%%%%%%SECTION%%%%%%%
\section{construction of special solutions}\label{sec: solutions_construction}
In this section, we describe the construction of the specific solutions of the Schr\"odinger equation that were used in \cite{zbMATH08122191} and \cite{zbMATH07801151}. These solutions will allow us to quantify the stability estimates. First we build the time-harmonic solutions that will be applied in the case of time-independent potentials, and then we turn our attention to the time-dependent case with the construction of the CGO solutions. 

%%%%%%%%SUBSECTION%%%%%%%%%%%%%
\subsection{Stationary states}\label{sec: stationary states}
In the case on time-independent potentials $V=V(x)$, one can make use of time-harmonic waves of the form
\[
\psi(t, x):= e^{-i\lambda^2t} w(x) \quad\text{in } \R\times \R^n.
\]
For $\psi$ to be a solution of the Schr\"odinger equation, the stationary state $w$ must be a solution of its stationary counterpart, also known as Helmholtz equation:
\[
(\Delta+\lambda^2-V)w=0, \quad  \text{in }  \R^n.
\]
We will follow the construction of these stationary states as perturbation of Herglotz waves, as performed in Section 2 of \cite{zbMATH08122191}. These were constructed in the functional space $B^*(\R^n)$ under the assumption $\vvvert V \vvvert < \infty$. As we explained in subsection \ref{subsec:potential_decay}, the slightly stronger assumption $\| V \|_\delta < \infty$ implies $\vvvert V \vvvert < \infty$, which guarantees that the arguments follow trivially.

For $\lambda>0$ and $f\in C^\infty(\Sph^{n-1})$, we define the \textit{Herglotz waves}
\begin{equation}\label{eq: Herglotz}
E_\lambda f(x) := \int_{\Sph^{n-1}} e^{-i\lambda x \cdot \theta} f(\theta) \, \dd\sigma(\theta) \quad \forall x\in \R^n,
\end{equation}
where $\sigma$ denotes the surface measure on $\Sph^{n-1}$. Lemma 2.1 in \cite{zbMATH08122191} guarantees that the Herglotz-wave belongs to $B^*(\R^n)$ and it holds that 
$(\Delta+\lambda^2)E_\lambda f=0$ in $\R^n$.  
In particular, the following estimate holds:
\begin{equation*}
\| E_\lambda f\|_{B^*(\R^n)} \lec_n \frac{1}{\lambda^{(n-1)/2}} \| f\|_{L^2(\Sph^{n-1})}
\end{equation*}
for all $f\in C^\infty(\Sph^{n-1})$ and $\lambda>0$. 

%For $\lambda>0$ sufficiently large we construct stationary states of the form $w(x) := E_\lambda f(x) + v(x)$, where $v$ is a perturbation of the Herglotz wave. 
%
Now, for $\lambda>0$ and $f\in \cal{S}(\R^n)$, we define the solution operator
\[
P_\lambda f(x) := \frac{1}{(2\pi)^{n/2}} \lim_{\ve \to 0} \int_{\{ \xi\in \R^n: | \lambda^2-|\xi|^2| > \ve \}} \frac{e^{i x \cdot \xi}}{\lambda^2-|\xi|^2} \wh{f}(\xi) \, \dd \xi \quad \forall x\in \R^n.
\]
Above and throughout the text, we denote with $\wh f$ the Fourier transform of $f$, with the following normalisation 
\[
\wh f(\xi):= \frac{1}{(2\pi)^{n/2}} \int_{\R^n} f(x) e^{-ix\cdot\xi} \, \dd x, \quad \forall \xi\in\R^n.
\]
For every $\lambda>0$ and $f\in B(\R^n)$ we also have that $(\Delta + \lambda^2) P_\lambda f= f$ in $\R^n$. 
Also, Lemma 2.2 in \cite{zbMATH08122191} gives us an estimate of the form
\begin{equation}\label{eq: Plambda_bound}
%\|P_\lambda f\|_{L^2_{-\delta}(\R^n)} \lec 
\| P_\lambda f \|_{B^*(\R^n)} \lec_n \frac{1}{\lambda} \| f\|_{B(\R^n)} .
%\lec \frac{1}{\lambda} \|f\|_{L^2_\delta(\R^n)}
\end{equation}

To obtain the desired stationary states as perturbations $w=E_\lambda f + v$, one may find correction terms $v=v(x)$ satisfying 
\[
(\Delta + \lambda^2)v = V(E_\lambda f +v).
\]
These correction terms might therefore be found as $v=\left[P_\lambda\circ V\right] (E_\lambda f +v)$, or in turn as
\[
v = \left[\Id-P_\lambda \circ V\right]^{-1} \left(P_\lambda \left(V E_\lambda f\right)\right)
\] 
As a result, the construction of $v$ depends on the invertibility of the operator $\Id-P_\lambda \circ V$ in the space $B^*(\R^n)$, and in the fact that $V E_\lambda f$ belongs to $B(\R^n)$. 
The first fact was proven via a Neumann series argument in \cite[Corollary 2.4]{zbMATH08122191} under the assumption $\vvvert V\vvvert<\infty$, for large enough $\lambda > C_n \vvvert V \vvvert + \eta$ and for any $\eta>0$, where $C_n$ denotes the constant in \eqref{eq: Plambda_bound}. The second one is a consequence of H\"older's inequality. For details on the whole argument, the reader is referred to section 2 in \cite{zbMATH08122191}

With this in hand we are able to wrap up the construction of stationary states. In particular for $V$ such that $\vvvert V \vvvert<
\infty$, $\lambda > \lambda_V := C_n \vvvert V \vvvert$ and $f\in C^\infty(\Sph^{n-1})$, set 
\[
u= E_\lambda f, 
\]
\[
v=(Id-P_\lambda\circ V)^{-1} [P_\lambda (Vu)]. 
\]
Then, $w= u+v \in B^*(\R^n)$ solves the equation 
\[
(\Delta + \lambda^2 - V)w=0 \,\, \, \textup{in} \,\,\, \R^n, 
\]
and, for every $\eta>0$, there exists a constant $C>0$, that depends only on $n$, such that 
\[
\|v\|_{B^*(\R^n)} \lec \frac{ C \lambda_V}{\lambda} \| u \|_{B^*(\R^n)}
\]
for all $\lambda \geq \lambda_V+\eta$. 
This is the content of \cite[Proposition 2.5]{zbMATH08122191}. Notice that the perturbations $v$ become negligible with respect to the leading terms $E_\lambda f$ as $\lambda $ grows to infinity.

Then,
\begin{equation}\label{eq: time-harmonic solutions}
\psi(t, x):= e^{-i\lambda^2t} w(x)=e^{-i\lambda^2t}(E_\lambda f (x) + v(x))
\end{equation}
are time-harmonic solutions in $C([0, T];B^*(\R^n))$ of the Schr\"odinger equation with potential $V=V(x)$. 
%
%We also have the dependence $\psi\equiv \psi_{\lambda}^V$ but throughout the rest of the article we drop notation.

%%%%

%%%%%%%%%%%SUBSECTION%%%%%%%%%%%%
\subsection{CGO solutions}\label{sect: CGO}
In this subsection, we will follow the construction of CGO solutions by Caro and Ruiz in \cite{zbMATH07801151}, which are suitable for the identification of time-dependent potentials. These solutions are of the form
\begin{equation}\label{eq:CGO}
	\mathbf{u}=e^\varphi(u^\sharp+u^\flat),
\end{equation}
where $\varphi(t,x)=i|\nu|^2t+\nu\cdot x$, for $\nu\in\R^n$. The leading term $e^\varphi\,u^\sharp$ will be chosen to be a solution to the free Schr\"odinger equation with potential $V=V(t,x)$ and will allow us to recover the Fourier transform, while the remainder $u^\flat$ will be constructed so that its norm decays as $|\nu|\to\infty$. From here on, we will denote with
\[
\wt{V}(t, x)=
\begin{dcases}
 V(t, x) \quad &(t, x) \in \Sigma \\
 0  &(t, x) \notin \Sigma
\end{dcases}
\]
the trivial extension of $V$ to $\R\times\R^n$. We will then construct the solutions \eqref{eq:CGO} to be a solution of the Schr\"odinger equation
\begin{equation}\label{eq: schrodinger_R}
	(i\partial_t + \Delta - \wt V) \mathbf{u} =0 \quad \textrm{in }\R\times\R^n.
\end{equation}
For the leading term, we need $u^\sharp$ to be a solution of the conjugated equation 
\[
(i\partial_t+\Delta+2\nu\cdot\nabla)u^\sharp=0 \quad \text{in } \R^n\times \R^n.
\]
One possible choice is to force $\nu\cdot \nabla u^\sharp=0$. For this, consider any $f\in L^2(\R^{n-1})$, and a rotation $Q\in O(n)$ such that $Qe_n=\nu/|\nu|$. We then define
\[
u^\sharp(t,x)=[e^{it\Delta'}\check f](Qe_1\cdot x,\ldots, Qe_{n-1}\cdot x),
\]
where $-\Delta'$ is the free Laplacian in $\R^{n-1}$ and $\check f$ denotes the inverse Fourier transform of $f$.
Equivalently, in Fourier terms, we can write 
\[
u^\sharp(t,x)=\frac{1}{(2\pi)^{(n-1)/2}}\int_{\R^{n-1}}e^{i Q \bs{\xi} \cdot x}e^{-it|\xi'|^2} f(\xi^\prime)\,d\xi', 
\]
where $\boldsymbol{\xi} = (\xi', 0)=(\xi_1, \cdots, \xi_{n-1}, 0)$. Notice that $|\xi'|=|\bs{\xi}|$. 
Since $\nu\cdot \nabla u^{\sharp}=0$  and $|Q \bs{\xi}| = |\xi'|$,
it is easy to check that $(i\partial_t+\Delta+2\nu\cdot\nabla)u^\sharp=0$, and therefore $(i\partial_t+\Delta)[e^\varphi u^\sharp]=0$. Now, define the dyadic strips in time
\[ \Upsilon_0 = \{ t \in \R : |t| \leq 1  \}, \quad
\Upsilon_j = \{ t \in \R : 2^{j - 1} < |t| \leq 2^j \} \enspace \forall j \in \N, 
\]
and recall the definition of the dyadic strips in space \eqref{eq: d_strips_space} which we restate here as 
\[ 
\Omega^\nu_0 = \{ x \in \R^n : |x \cdot Qe_n| \leq 1  \}, \quad
\Omega^\nu_k = \{ x \in \R^n : 2^{k - 1} < |x \cdot Qe_n| \leq 2^k \}, \enspace \forall k \in \N.
\]
Denoting with
\[
\Pi_\alpha^\nu = \Upsilon_j^\nu\times\Omega_k^\nu \quad \forall \alpha = (j,k)\in \N^2_0,
\]
the sequence of dyadic space-time strips in direction $\nu$, we can easily check that there exists $C>0$ such that
\begin{equation}\label{eq: sol_sharp_est1}
\| u^\sharp \|_{L^2 (\Pi_\alpha^\nu)} \leq C 2^{|\alpha| /2} \| f \|_{L^2(\R^{n - 1})},
\end{equation}
which in particular means that
\begin{equation}\label{eq:solSharpEst}
\sup_{k \in \N_0} \big[ 2^{-k/2} \| u^\sharp \|_{L^2 (\Sigma^\nu_k)} \big] \lesssim_T \| f \|_{L^2(\R^{n - 1})}.
\end{equation}
Here we denoted with $|\alpha|=k+j$ and also remember the definition of $\Sigma_k^\nu$ in \eqref{eq: strips}.
Observe that, in order to obtain solutions to \eqref{eq: schrodinger_R} of the form \eqref{eq:CGO}, we need the remainder $u^\flat$ to satisfy
\begin{equation}\label{eq: neumann}
	(i\partial_t+\Delta+2\nu\cdot\nabla- \wt{V})u^\flat=\wt{V}\,u^\sharp,\quad \textrm{in }\R\times\R^n.
\end{equation}
This correction term can again be constructed by a Neumann series argument, similar to that outlined in subsection \ref{sec: stationary states}, by obtaining a suitable estimate for the conjugated operator $i\partial_t+\Delta+2\nu\cdot\nabla$. The Neumann argument works for potentials $V\in L^\infty(\Sigma)$ satisfying the decay condition \eqref{eq: polynomial_decay_V_t}, and one can estimate the size of the remainder as
\begin{equation}\label{eq: sol_flat_est1} 
|\nu|^{1/2} \sup_{\alpha \in \N^2_0} \big[ 2^{-|\alpha|(1/2 - \theta)} \| u^\flat \|_{L^2 (\Pi_\alpha^\nu)} \big] \leq C  \| f \|_{L^2(\R^{n - 1})}, 
\end{equation}
for some $C>0$ and, in particular
\begin{equation}\label{eq:solFlatEst}   
|\nu|^{1/2} \sup_{k \in \N_0} \big[ 2^{-k(1/2 - \theta)} \| u^\flat \|_{L^2 (\Sigma^\nu_k)} \big] \lesssim_T  \| f \|_{L^2(\R^{n - 1})}. 
\end{equation}
It is easy to see that the super-exponential decay condition for the potential readily implies \eqref{eq: polynomial_decay_V_t}. 
We consider convenient to refer the reader to \cite{zbMATH07801151} for the details of the Neumann series argument, with the definition of the proper Banach spaces in which this is performed. Our statements above are a direct consequence of Theorem 2.9 in \cite{zbMATH07801151}.

At this point, we should also note that the estimates \eqref{eq:solSharpEst} and \eqref{eq:solFlatEst} can be expressed in terms of $L^2_TL^2_{-\delta}$-norms. Indeed, since $\nu=|\nu| Qe_n$ and in the strips $\Omega^\nu_k$ it holds $ \frac{1}{|x|} \leq \frac{|\nu|}{|x\cdot\nu|} \lec 2^{-k} $, we estimate
\[
\begin{aligned}
\|u^\sharp\|^2_{L^2_TL^2_{-\delta}} &= \int_{\Sigma} \langle x\rangle ^{-2\delta}|u^\sharp|^2
=\sum_{k\in\N_0} \int_{0}^T\int_{\Omega^\nu_k} \langle x\rangle ^{-2\delta}|u^\sharp|^2 
\lec \sum_{ k \in\N_0} \int_{0}^T \int_{\Omega^\nu_k} 2^{-k2\delta} |u^\sharp|^2 \\
&= \sum_{ k \in\N_0} 2^{-k (2\delta-1)} 2^{-k} \| u^\sharp \|_{L^2(\Sigma^\nu_k)}^2
\lec \sup_{k\in\N_0}\big[ 2^{-k}\| u^\sharp \|_{L^2(\Sigma^\nu_k)}^2 \big],
\end{aligned}
\]
since $\delta>1/2$. The above estimate and \eqref{eq:solSharpEst} imply that 
\begin{equation}\label{eq:solSharpEst2}
	\|u^\sharp\|_{L^2_TL^2_{-\delta}}\lec \sup_{k\in\N_0}\big[ 2^{-k/2}\| u^\sharp \|_{L^2(\Sigma^\nu_k)} \big] 
	 \lec_T  \| f \|_{L^2(\R^{n - 1})}.
\end{equation}
The same argument can be used along with \eqref{eq:solFlatEst} to show that, for $|\nu|$ big enough, 
\begin{equation}\label{eq:solFlatEst2}
	\|u^\flat\|_{L^2_TL^2_{-\delta}} \lec_T |\nu|^{-1/2} \| f \|_{L^2(\R^{n - 1})},
\end{equation}
with this inequality being valid for all $\delta>0$.
%

%%%%%%%SECTION%%%%%%%SECTION%%%%%%%%%SECTION%%%%%%%%
\section{Density of solutions}\label{sec: density}
In this section, we provide some density results that are crucial for our approach, since they allow for the distance to be used in the stability estimates. These are Runge approximation type results, showing that the corresponding special solutions that were
constructed in section \ref{sec: solutions_construction} can be approximated by physical solutions of the Schr\"odinger equation in the relevant weighted $L^2$ topologies.

We separate the density results for the case of time-independent and time-dependent potentials. Since a big part of the argument is identical we will only provide a rigorous proof for the more complicated case of time-dependent potentials. 

%%%%%%%%SUBSECTION%%%%%%%%%%%
\subsection{Time-dependent potentials}\label{subsec: approx_tdpotential}
In this case we will approximate CGO solutions by physical solutions in the $L^2_{w_{\delta, \ve}}$-topology, by fixing the weight $w_{\delta, \ve}(x) = e^{-2|x|^{1+\ve}}\langle x \rangle^{-2\delta}$. We consider potentials $V=V(t, x) \in L^1((0, T); L^\infty(\R^n))$ with $\|V\|_{w_{\delta, \ve}} < \infty$. See section \ref{subsec: spaces} for the definition of the spaces above.
For $\nu\in \R^n$, define the linear space of solutions
\[
\mathcal C_{V, \nu} := \left\{u\in L^2_{\textrm{loc}}(\R^{n+1}):(i\partial_t+\Delta- \wt{V})u=0,\; \sup_{\alpha \in \N_0^2} \big[ 2^{-|\alpha|/2} \| e^{- \nu \cdot \x} u \|_{L^2(\Pi_\alpha^\nu)} \big] <\infty\right\}, 
\]
and let
\begin{equation}\label{eq: space_exponential}
\mathcal C_V := \bigcup_{\nu\in\R^n}\mathcal C_{V, \nu}
\end{equation}
We will drop the notation to $\cal C$, whenever there is no ambiguity about the potential.
It is easy to see that CGO solutions of the form \eqref{eq:CGO} belong to $\mathcal C$. Indeed, if $\nu\in\R^n$ and 
\[
	\mathbf{u}=e^\varphi(u^\sharp+u^\flat),
\]
with $\varphi(t,x)=i|\nu|^2t+\nu\cdot x$ and $u^\sharp$, $u^\flat$ as constructed in Section \ref{sect: CGO}, then, by virtue of the estimates \eqref{eq: sol_sharp_est1} and \eqref{eq: sol_flat_est1} we have that
\[
\sup_{\alpha \in \N_0^2} \big[ 2^{-|\alpha|/2} \| e^{- \nu \cdot \x} \mathbf{u} \|_{L^2(\Pi_\alpha^\nu)}\big]\leq\sup_{\alpha \in \N_0^2} \big[ 2^{-|\alpha|/2} \| u^\sharp \|_{L^2(\Pi_\alpha^\nu)}\big] + \sup_{\alpha \in \N_0^2} \big[ 2^{-|\alpha|/2} \| u^\flat \|_{L^2(\Pi_\alpha^\nu)}\big]<\infty.
\]
Furthermore, we have the following inclusion.
%
%%%%%%%%LEMMA%%%%%%%%%%%%LEMMA%%%%%%%%%%%%%%%
\begin{lemma}\label{lem: C_subspace_H}
	The space $\mathcal C$
	 is a subspace of the Banach space $L^2_{w_{\delta, \ve}}(\Sigma)$.
\end{lemma}
\begin{proof}
Start by observing that, for $\nu,x\in\R^n$,
\[
-|x|^{1+\ve}+\nu\cdot x\leq-|x|^{1+\ve}+|\nu|\, |x|\leq \ve |\nu|^{(1+\ve)/\ve}. 
\]
The last inequality is a consequence of the fact that the function $g(z)=-z^{1+\ve}+|\nu|z$ for $z\geq 0$, attains its maximum at $z_0=\left(\frac{|\nu|}{1+\ve}\right)^{1/\ve}$ with $g(z_0)=\ve\left(\frac{|\nu|}{1+\ve}\right)^{(1+\ve)/\ve}\leq \ve |\nu|^{(1+\ve)/\ve}$. 
Consequently,
\begin{equation}\label{eq: control_e_weight}
e^{-|x|^{1+\ve}}%\leq e^{\ve |\nu|^{(1+\ve)/\ve}} e^{-|\nu| |x|}
\leq e^{\ve |\nu|^{(1+\ve)/\ve}} e^{-\nu\cdot x}, 
\end{equation}
which implies that
\[
\|u\|_{L^2_{w_{\delta, \ve}}(\Sigma)}\leq e^{\ve |\nu|^{(1+\ve)/\ve}}\|\langle x\rangle^{-\delta}e^{-\nu\cdot x}u\|_{L^2(\Sigma)}.
\]
Also, since $\delta>1/2$, it holds that $\langle x\rangle^{-2\delta}\lesssim 2^{-2k\delta}$ in $\Omega_k^\nu$ for $k\in\N$, so that
\[
\begin{aligned}
\|u\|_{L^2_{w_{\delta, \ve}}(\Sigma)}^2&\lesssim_\nu \sum_{ k \in\N_0} 2^{-k (2\delta-1)} 2^{-k} \| e^{-\nu\cdot x}u \|_{L^2(\Sigma^\nu_k)}^2
\lec_{\nu,\delta}
\sup_{k\in\N_0}\big[ 2^{-k}\| e^{-\nu\cdot x}u \|_{L^2(\Sigma^\nu_k)}^2 \big].
\end{aligned}
\]
Finally, since $T<\infty$, there exists $j_0\in\N_0$ so that $2^{j_0-1}< T \leq 2^{j_0}$. Then $(0, T) \subset \cup_{j=0}^{j_0} \Upsilon_j$ and using Fubini, for any $k\in\N_0$ it holds that
\[
\| e^{-\nu \cdot x} u \|^2_{L^2(\Sigma^\nu_k)} = \int_{\Omega_k^\nu} |e^{-\nu\cdot x}|^2 \sum_{j=0}^{j_0} \|u_j\|_{L^2(\Upsilon_j)}^2
\lec_T \sum_{j=0}^{j_0} 2^{-j} \| e^{-\nu\cdot x} u\|^2_{L^2(\Upsilon_j\times\Omega_k^\nu)}
\lesssim \sup_{j\in\N_0}2^{-j}\|e^{-\nu\cdot x}  u\|^2_{L^2(\Upsilon_j\times\Omega_k^\nu)}.
\]
Bringing all pieces together, we get that
\[
\|u\|_{L^2_{w_{\delta, \ve}}(\Sigma)}\lesssim_{\nu,\delta,T} \sup_{\alpha \in \N_0^2} \big[ 2^{-|\alpha|/2} \| e^{- \nu \cdot \x} u \|_{L^2(\Pi_\alpha^\nu)} \big] < \infty.
\]
\end{proof}
We remark that by means of the estimates \eqref{eq:solSharpEst2} and \eqref{eq:solFlatEst2}, along with the inequality \eqref{eq: control_e_weight}, we can also get the following useful bound  
	\begin{equation}\label{eq: CGO_solution_in_H_norm}
	\|\mathbf{u}\|_{L^2_{w_{\delta, \ve}}(\Sigma)} \leq e^{\ve |\nu|^{(1+\ve)/\ve}} (\|u^\sharp\|_{L^2_TL^2_{-\delta}} + \|u^\flat\|_{L^2_TL^2_{-\delta}}) 
	\lec_T e^{\ve |\nu|^{(1+\ve)/\ve}} \|f\|_{L^2(\R^{n-1})}.
	\end{equation}

An important ingredient for the proof of the density result is an integration-by-parts formula for exponentially growing solutions due to Caro and Ruiz \cite{zbMATH07801151}, which we restate here for reader's convenience. 
%
%%%%%PROPOSITION%%CR%%%%%%
\begin{proposition}\cite[Proposition 4.9]{zbMATH07801151}\label{prop:CGO-PhS_int-by-parts}\sl 
Consider $F \in L^2(\Sigma)$ such that 
there exists $c > 0$ so that $e^{c|\x|} F \in L^2(\Sigma)$.
Let $u \in C([0,T]; L^2(\R^n))$ satisfy the conditions
\[
\left\{
		\begin{aligned}
		& (i\partial_\tm + \Delta) u = F & & \textnormal{in} \enspace \Sigma, \\
		& u(0, \centerdot) = u(T, \centerdot) = 0 &  & \textnormal{in} \enspace \R^n.
		\end{aligned}
	\right.
\]
For every $\nu \in \R^n \setminus \{ 0 \}$ with $|\nu| < c$,
consider a measurable function 
$v^\nu : \R \times \R^n \rightarrow \C$ such that
\[ \sup_{\alpha \in \N_0^2} \big[ 2^{-|\alpha|/2} \| e^{- \nu \cdot \x} v^\nu \|_{L^2(\Pi_\alpha^\nu)} \big] + \sum_{\alpha \in \N_0^2} 2^{|\alpha| \theta} \| e^{- \nu \cdot \x} (i\partial_\tm + \Delta) v^\nu \|_{L^2(\Pi_\alpha^\nu)} < \infty, \]
for some $\theta \in (0, 1/2)$. Then,
\[\int_\Sigma (i\partial_\tm + \Delta) u \overline{v^\nu} = \int_\Sigma u \overline{(i\partial_\tm + \Delta)v^\nu} \, \]
for all $\nu \in \R^n \setminus \{ 0 \}$ with $|\nu| < c$.
\end{proposition}
%%%%%
%
We now proceed to the main result of this subsection, which is summarised in the next lemma.

%%%%LEMMA%%%%%LEMMA%%%%%%%%
\begin{lemma}\label{lem: density_of_solutions_time}
For the spaces of solutions \eqref{eq: space_physical_solutions} and \eqref{eq: space_exponential}, it holds that 
\[
\cal{C} \subset \ovl{\cal{P}}^{\|\cdot\|_{L_{w_{\delta, \ve}}^2}}.
\]
This is, physical solutions approximate solutions in $\cal C$, in the topology of $L_{w_{\delta, \ve}}^2(\Sigma)$. In particular, physical solutions approximate the CGO solutions in $L_{w_{\delta, \ve}}^2(\Sigma)$.
\end{lemma}
%%%%%%
%

\begin{proof}%[Proof of Lemma \ref{lem: density_of_solutions}]
Observe first that $\ovl{\cal{P}}^{\|\cdot\|_{L_{w_{\delta, \ve}}^2}} $ is a closed subspace of the Hilbert space $L_{w_{\delta, \ve}}^2(\Sigma)$. 
As a result, it holds that 
\[
\big({\cal{P}^\perp}\big)^{\perp} = \ovl{\cal{P}}^{\|\cdot\|_{L_{w_{\delta, \ve}}^2}},
\]
where the orthogonal complement is taken with respect to the inner product in $L_{w_{\delta, \ve}}^2(\Sigma)$, as defined in \eqref{eq: inner_product_weighted}; this is,
\[
\cal{P}^{\perp} = 
\Big\{ v\in L_{w_{\delta, \ve}}^2(\Sigma): \,\, \langle v, u \rangle_{w_{\delta, \ve}}=0 \quad \forall u\in C_T L^2 \,\, \textup{such that} \,\, (i\d_t +\Delta - V)u=0 \Big\}, 
\]
and
\[
\big({\cal{P}}^{\perp}\big)^{\perp} = 
\Big\{ g\in L_{w_{\delta, \ve}}^2(\Sigma) : \, \langle g, v\rangle = 0 , \quad \forall v\in \mathcal P ^\perp \Big\}.
\]
We claim that 
\[
\cal{C} \subset \big({\cal{P}}^{\perp}\big)^{\perp}.
\]
Let's start by characterising the orthogonal complement of $\mathcal P$. Indeed, for any $v\in \cal{P}^{\perp}$, we have that 
\begin{equation}\label{eq: inner_product_zero}
0=\langle v, u\rangle_{w_{\delta, \ve}}=\int_\Sigma w_{\delta,\ve}(x)\,v(t,x)\,\overline{u(t,x)}\, \dd t \dd x
\end{equation}
for any physical solution $u\in C_TL^2$ of the Schr\"odinger equation. Set now $G(t, x)\defeq w_{\delta,\ve}(x)\,v(t, x)$ for $v\in \cal{P}^{\perp}$. 
We verify that $G\in (L^2_{w_{\delta, \ve}} (\Sigma))^* =L^2_{w_{\delta, \ve}^{-1}}(\Sigma)$. Indeed,
\[
\int_{\Sigma}  w_{\delta,\ve}^{-1}(x)\,|G(t, x)|^2 \, \dd t \, \dd x
= \int_{\Sigma}w_{\delta,\ve}(x)\, |v(t, x)|^2 \dd t \, \dd x = \|v\|_{L^2_{w_{\delta, \ve}}(\Sigma)}< \infty. 
\]
In particular, since $w_{\delta,\ve}(x)\leq 1$, this implies that $G\in L^1((0, T); L^2(\R^n))$, so that there exists a unique solution $\Phi\in C([0, T]; L^2(\R^n))$ to the following artificial problem: 
\begin{equation}\label{eq: art_prob}
\begin{dcases}
(i\d_t+\Delta - \ovl{V})\Phi = G, \quad &\textup{in}\,\, \Sigma, \\
\Phi(T, \cd) = 0, &\textup{in}\,\, \R^n.
\end{dcases} 
\end{equation}
Then, \eqref{eq: inner_product_zero} can equivalently be written as $\displaystyle \int_\Sigma \ovl{G} u = 0$ for all physical solutions $u\in C_TL^2$. Using \cite[Lemma 4.5]{zbMATH07801151}, we conclude that $\Phi(0, \cd)=0$. 
Now, letting $\psi\in\mathcal C$ we have that
%%%%
\begin{equation}\label{eq: int_by_parts}
\begin{aligned}
\langle v, \psi\rangle_{w_{\delta, \ve}} &= \int_\Sigma  w_{\delta,\ve}(x)\,v(t, x)\overline{\psi(t, x)}  \, \dd t  \dd x \\
&= \int_\Sigma \ovl{\psi} G = \int_\Sigma (i\d_t+\Delta - \ovl{V}) \Phi \ovl{\psi} 
= \int_\Sigma \Phi \ovl{(i\d_t+\Delta - V)\psi} = 0.
\end{aligned}
\end{equation}
%%%%
%
In the second to last equality we have used the integration by parts of Proposition \ref{prop:CGO-PhS_int-by-parts}. In the last equality we used the fact that $\psi$ is a solution of the Schr\"odinger equation. Hence, we get that $\psi \perp v$
for any $v\in\cal{P}^{\perp}$, and therefore we conclude that $\psi\in \big(\mathcal P ^\perp)^\perp=\ovl{\cal{P}}^{L^2_{w_{\delta, \ve}}}$.

To justify the use of Proposition \ref{prop:CGO-PhS_int-by-parts}, on the one hand notice that, for any $c>0$ we have $e^{c|x|}(i\partial_t+\Delta)\Phi =e^{c|x|} (G+\overline V \Phi) \in L^2(\Sigma)$. The fact $e^{c|x|}\ovl{V} \Phi \in L^2(\R^n)$ is trivial due to  the super-exponential decay of $V$ along with $\Phi\in L^2(\Sigma)$. For the remaining term, first observe that $e^{c|x|} \leq e^{|x|^{1+\ve}}$ whenever $|x|\geq c^{1/\ve}$. Then we estimate 
\begin{align*}
\int_{\Sigma} |e^{c|x|} G|^2 &\leq \int_\Sigma |\langle x\rangle^{-2\delta} e^{-2|x|^{1+\ve}} e^{c|x|}v |^2 \\
&\leq e^{2c^{1/\ve +1 }} \int_{(0, T)\times \{ |x| < c^{1/\ve} \}} | \langle x\rangle^{-2\delta} e^{-2|x|^{1+\ve}} v|^2 
+ \int_{(0, T)\times \{ |x| \geq c^{1/\ve}\}} |\langle x \rangle^{-2\delta} e^{-|x|^{1+\ve}} v|^2 \\
&\leq e^{2c^{1/\ve +1 }} \int_\Sigma \langle x \rangle^{-2\delta} e^{-2|x|^{1+\ve}} |v|^2 
= e^{2c^{1/\ve +1 }} \|v\|_{L^2_{w_{\delta, \ve}}} <\infty,
\end{align*}
since  $v \in L^2_{w_{\delta, \ve}}$.

On the other hand, since $\psi \in \cal C$, we have  $\psi \in \mathcal C_\nu$ for some $\nu\in \R^n$ and $(i\partial_t+\Delta)\psi = \wt{V}\psi$. Recall that $\wt{V}$ denotes the trivial extension of $V$ to $\R\times\R^n$. 
Thus
\begin{align*}
\sum_{\alpha \in \N_0^2} 2^{|\alpha| \theta} \| e^{- \nu \cdot \x} (i\partial_\tm + \Delta) \psi \|_{L^2(\Pi_\alpha^\nu)} &= \sum_{\alpha \in \N_0^2} 2^{|\alpha| \theta} \| e^{- \nu \cdot \x} \wt{V} \psi \|_{L^2(\Pi_\alpha^\nu)} \\
&\leq \sup_{\alpha \in \N_0^2} \big[ 2^{-|\alpha|/2} \| e^{- \nu \cdot \x} \psi \|_{L^2(\Pi_\alpha^\nu)} \big] \Bigg( \sum_{\alpha \in \N_0^2} 2^{|\alpha| (\theta+1/2)} \| \wt{V}  \|_{L^\infty(\Pi_\alpha^\nu)} \Bigg) <\infty,
\end{align*}
where we have used that
\[
\sum_{\alpha \in \N_0^2} 2^{|\alpha| (\theta+1/2)} \| \wt{V}  \|_{L^\infty(\Pi_\alpha^\nu)} \lesssim_T \sum_{k\in\N_0} 2^{k (\theta+1/2)} \| V \|_{L^\infty(\Sigma_k^\nu)} \leq \| |x|V\|_{L^\infty(\Sigma)} \sum_{k\in\N_0} 2^{k (\theta+1/2)}2^{-k} < \infty
\]
for all $\theta\in(0,1/2)$. The norm $ \| |x|V\|_{L^\infty(\Sigma)}$ above is finite due to super-exponential decay. One could alternatively use directly the decay \eqref{eq: polynomial_decay_V_t}. This ends the proof of the density lemma.
\end{proof}
%%%%%%%%%%%%
%

%%%%%%SUBSECTION%%%%%%%%%
\subsection{Time-independent potentials}\label{subsec:approx_tindep}
We turn now our attention to the case of potentials $V\in L^\infty(\R^n)$ with $\|V\|_{w_\delta} < \infty$. In this case we will approximate time-harmonic solutions of the Schr\"odinger equation by physical solutions in the $L^2_{w_\delta}$-topology. Recall from subsection \ref{subsec: spaces} that $w_\delta(x)= \langle x \rangle^{-2\delta}$.

To that end, we introduce the following space of more general solutions to the Schr\"odinger equation
\begin{equation}\label{eq: space_time_harmonic}
\cal{S}_V := \{ \psi \in C([0, T]; B^*(\R^n)) :  (i\partial_t+\Delta-V) \psi = 0 \}.
\end{equation}
Obviously, the time-harmonic solutions $\psi(t,x) = e^{-i\lambda^2t}w(x)$ where $w\in B^*(\R^n)$ with $ (\Delta + \lambda - V)w =0$ and $\lambda>0$, belong to $\cal{S}_V$. Recall \eqref{eq: time-harmonic solutions} for the definition of these kind of solutions. When there is no ambiguity about the potential, we will just write $\cal S$. The density result is the following: 
%
%%%%%%LEMMA%%%%LEMMA%%%%%%%%
\begin{lemma}\label{lem: density_of_solutions_t_independent}
For the spaces of solutions \eqref{eq: space_physical_solutions} and \eqref{eq: space_time_harmonic}, it holds that 
\[
\cal{S} \subset \ovl{\cal{P}}^{\|\cdot\|_{L^2_{w_\delta}}}.
\]
This is, physical solutions approximate $B^*(\R^n)$-solutions in the topology of $L^2_{w_\delta}$ and in particular, they approximate stationary-state solutions in this topology.
\end{lemma}

\begin{proof}
We give a brief description of the proof since it is almost identical to the proof of the Lemma \ref{lem: density_of_solutions_time}. Following the same argument, here using the space $L^2_{w_\delta}$ instead of $L^2_{w_{\delta, \ve}}$, we can again end up to the same artificial problem \eqref{eq: art_prob} and conclude $\Phi(0, \cd)=0$. The only difference now is to prove that the analogous inner product to \eqref{eq: int_by_parts} is zero. Here, the inner product is with respect to the weight $w_\delta(x) = \langle x \rangle^{-2\delta}$. Since $\Phi \in C_T L^2$ and for $\psi\in \cal S$ we have $\psi\in C_TB^*$, we can again integrate by parts and conclude 
\[
\int_\Sigma (i\d_t+\Delta -\ovl{V}) \Phi \ovl{\psi} = \int_\Sigma \Phi \ovl{(i\d_t+\Delta - V) \psi} = 0.
\]
This integration by parts is directly justified by \cite[Lemma 3.2]{zbMATH08122191}. This concludes the proof of the lemma. 
\end{proof}
%%%%%%%%%%%
%%

\subsection{Useful bounds due to the density of solutions}
The next two lemmas are used in the final estimates of Section \ref{sec: stability_estimates}. We state them here as separate results for reader's convenience. 
The main component for their proof is the density of solutions, as proved in Lemmas \ref{lem: density_of_solutions_time} and \ref{lem: density_of_solutions_t_independent}. Once again it becomes clear the necessity of the assumptions $\| V \|_{w_{\delta, \ve }}< \infty$ and 
$\|V\|_{w_\delta} < \infty$. 
%%%
%

%%%%%%%LEMMA%%%%%%%LEMMA%%%%%%
\begin{lemma}\label{lem: bound_integral_CGO_by_distance}
Let $n\geq2$ and $V_1, V_2\in L^1((0, T); L^\infty(\R^n))$ with $\|V_j\|_{w_{\delta,\ve} }< \infty$, for $j\in \{1, 2\}$. Let ${\bf u_1}$ and ${\bf u_2}$ be the CGO solutions \eqref{eq:CGO} to the Schr\"odinger equation, as constructed in subsection \ref{sect: CGO}, with potentials $V_1$ and $V_2$ respectively. Then, there holds 
\begin{equation}\label{eq: integral_bound_CGO_by_dist}
\Big| \int_\Sigma (V_1-V_2) \mathbf{u}_1 \ovl{\mathbf{u}_2} \Big| 
\leq  
\dist_{\delta, \ve} (\cal{U}_T^1, \cal{U}_T^2) \| \mathbf{u}_1 \|_{L^2_{w_{\delta, \ve}}(\Sigma)} \| \mathbf{u}_2 \|_{L^2_{w_{\delta, \ve}}(\Sigma)}. 
\end{equation}
\end{lemma}

%%%%%PROOF%%%%%PROOF%%%%%%%
\begin{proof}
In the virtue of the Runge approximation of Lemma \ref{lem: density_of_solutions_time}, we construct sequences of physical solutions $\{u_1^n\}_{n\in\N}$ and $\{u_2^n\}_{n\in\N}$ such that $\| {\bf u_1} -u^n_1\|_{L^2_{w_{\delta, \ve}}}\xrightarrow{n\to\infty}0$ and $\| {\bf u_2} -u_2^n\|_{L^2_{w_{\delta, \ve}}} \xrightarrow{n\to\infty}0$. We estimate as follows
\begin{align*}
\Big| \int_\Sigma (V_1 - V_2) {\bf u_1} \ovl{ {\bf u_2}} \Big| &\leq \Big| \int_\Sigma (V_1-V_2)({\bf u_1} \ovl{{\bf u_2}} - u^{n}_1\ovl{u^{n}_2}) \Big|
+ \Big| \int_\Sigma (V_1 - V_2) u^{n}_1 \ovl{u^{n}_2} \Big| \\
&= \Big| \int_\Sigma (V_1-V_2)\big[({\bf u_1} - u^{n}_1)\ovl{{\bf u_2}} + u^{n}_1(\ovl{{\bf u_2}}-\ovl{u^{n}_2}) \big] \Big| 
+ \Big| \int_\Sigma (V_1-V_2) u^{n}_1 \ovl{u^{n}_2} \Big| \\
&\leq \| V_1-V_2 \|_{w_{\delta, \ve}} \big( \| {\bf u_1} - u^{n}_1\|_{L^2_{w_{\delta, \ve}}} \| {\bf u_2} \|_{L^2_{w_{\delta, \ve}}} + \|u^{n}_1\|_{L^2_{w_{\delta, \ve}}}\|{\bf u_2} - u^{n}_2 \|_{L^2_{w_{\delta, \ve}}} \big) \\
&+ \dist_{\delta, \ve}(\cal{U}_T^1, \cal{U}_T^2) \|u^n_1\|_{L^2_{w_{\delta, \ve}}} \|u^n_2\|_{L^2_{w_{\delta, \ve}}} \\
&\leq \| V_1-V_2 \|_{w_{\delta, \ve}} \big( \| {\bf u_1} - u^n_1\|_{L^2_{w_{\delta, \ve}}} \|{\bf u_2}\|_{L^2_{w_{\delta, \ve}}} + ( \|u^n_1- {\bf u_1}\|_{L^2_{w_{\delta, \ve}}} + \| {\bf u_1} \|_{L^2_{w_{\delta, \ve}}})\|{\bf u_2} - u^n_2 \|_{L^2_{w_{\delta, \ve}}} \big) \\
&+  \dist_{\delta, \ve}(\cal{U}_T^1, \cal{U}_T^2) \big( \|u^n_1- {\bf u_1}\|_{L^2_{w_{\delta, \ve}}} + \|{\bf u_1}\|_{L^2_{w_{\delta, \ve}}} \big) 
\big(\|u^n_2 - {\bf u_2}\|_{L^2_{w_{\delta, \ve}}} + \|{\bf u_2}\|_{L^2_{w_{\delta, \ve}}} \big),
\end{align*}
where in the second to last inequality we used the definition of the distance \eqref{eq: distance}. By sending $n\to\infty$, we find the estimate \eqref{eq: integral_bound_stationary_by_dist}.
\end{proof}

%%%%%%%%LEMMA%%%%%LEMMA%%%%%%%%
\begin{lemma}\label{lem: bound_int_t-harmonic_by_distance}
Let $n\geq2$ and $V_1, V_2\in L^1(\R^n)\cap L^\infty(\R^n)$ with $\|V_j\|_{w_\delta} < \infty$, for $j\in \{1, 2\}$. Let also $\psi_1\equiv \psi_\lambda^{V_1}$ and $\psi_2 \equiv \psi_\lambda^{V_2}$ be the time-harmonic solutions \eqref{eq: time-harmonic solutions} to the Schr\"odinger equation, as constructed in subsection \ref{sec: stationary states} for potentials $V_1$ and $V_2$ respectively. Then, there holds
\begin{equation}\label{eq: integral_bound_stationary_by_dist}
\Big| \int_\Sigma (V_1-V_2) \psi_1 \ovl{\psi_2} \Big| \leq T \dist_\delta (\cal{U}_T^1, \cal{U}_T^2) \|\psi_1\|_{C_TL^2_{-\delta}} \| \psi_2\|_{C_TL^2_{-\delta}}. 
\end{equation}
\end{lemma}
\begin{proof}
We omit the details of the proof, since it follows by the exact same arguments used in the proof of the Lemma \ref{lem: bound_integral_CGO_by_distance}, using the density of the Lemma \ref{lem: density_of_solutions_t_independent} instead and the fact that here $\|\psi_j\|_{L^2_{w_\delta}} \leq T^{1/2} \| \psi_j\|_{C_TL^2_{-\delta}}$ for $j\in\{1, 2\}$. 
\end{proof}
%%%%%%%%%%%%

%%%%%%%%SECTION%%%%%%%%%SECTION%%%%%%%%%SECTION%%%%%%%%%%%
\section{Stability estimates}\label{sec: stability_estimates}
In this final section we provide the proofs of the Theorems \ref{th: stability_t_independent} and \ref{th: stability_t_dependent}. We start with the case of time-independent potentials and the proof of  Theorem \ref{th: stability_t_independent}, and the same intuition carries out for the proof of Theorem \ref{th: stability_t_dependent}.

We note that the $H^{-1}(\R^n)$ norm can be defined, for $F\in \mathcal{S}'(\R^n)$, via the Fourier transform as
\[
 \|F\|_{H^{-1}(\R^n)}^2=\int_{\R^n}(1+|\xi|^2)^{-1/2}|\widehat F(\xi)|^2\dd \xi.
\]
%%%%%%%SUBSECTION%%%%%%%%%%%%%%%
%
\subsection{Time-independent potentials}\label{ssec: stability_time_independent}
%
%Recall that for $\delta>1/2$ we use the weight $w_\delta(x)=\langle x \rangle^{-\delta}$.
%
We start this section by constructing a sequence of densities that will allow us to extract information on the Fourier transform of $V_1-V_2$ from the Hergoltz waves as defined in \eqref{eq: Herglotz}.
This construction is identical as that of section 4.1 in \cite{zbMATH08122191}.

Indeed, consider $\chi \in \cal{D}(\R^n)$ such that $0\leq \chi(\xi) \leq 1$ for all $\xi \in \R^n$ and $\supp \chi \subset \{\xi\in\R^n : |\xi| < 1/2\}$. For 
$\ve>0$, define 
\[
\chi_\ve(\xi) := \frac{1}{\ve^{n-1}} \chi \Big( \frac{\xi}{\ve}, \frac{\xi_n - 1}{\ve^2} \Big) \quad \forall \xi\in \R^n. 
\] 
Here we use the notation $\xi = (\xi \cdot e_1, \cdots, \xi \cdot e_{n-1})$ and $\xi_n = \xi \cdot e_n$ with $\{e_1, \cdots, e_n\}$ denoting the standard basis of $\R^n$. We restrict this function to $\Sph^{n-1}$ and denote it by
\[
f_\ve^{e_n} = \chi_\ve|_{\Sph^{n-1}} .
\]
The parabolic scaling in the definition of $\chi_\ve$ ensures that the support of $f_\ve^{e_n}$ is within a spherical cap on $\Sph^{n-1}$ of radius $\ve$ and centred at $e_n$. If $Q\in \mathrm{O}(n)$, we also define 
\[
f_\ve^{Q e_n}(\theta) = f_\ve^{e_n} (Q^\intercal \theta) \quad \forall \theta \in \Sph^{n-1}. 
\]
For every $Q\in \mathrm{O}(n)$ it holds that 
\[
\lim_{\ve \to 0} \| f_\ve^{Q e_n} \|_{L^1(\Sph^{n-1})} = \int_{\R^{n-1}} \chi(\eta, -|\eta|^2/2) \, d\eta
\]
and
\[
\| f_\ve^{Q e_n} \|_{L^2(\Sph^{n-1})} \leq \frac{C}{ \ve^{(n-1)/2}}
\]
for all $\ve\in(0, 1]$ and $Q\in\mathrm{O}(n)$. The constant in the last inequality only depends on the dimension $n$. A proof of these facts can be found in \cite[Lemma 4.1]{zbMATH08122191}.

Given $\kappa \in \R^n$ consider $\nu \in \Sph^{n-1}$ such that $\kappa \cdot \nu = 0$. For $\lambda \geq |\kappa|/2$ we define 
\[
\om_1= \frac{1}{\lambda}\frac{\kappa}{2} + \Big( 1 - \frac{|\kappa|^2}{4\lambda^2}\Big)^{1/2} \nu, 
\]
\[
\om_2= \frac{1}{\lambda}\frac{\kappa}{2} - \Big( 1 - \frac{|\kappa|^2}{4\lambda^2}\Big)^{1/2} \nu.
\]
Note that $\om_1, \, \om_2 \in \Sph^{n-1}$ and since any point on the sphere can be written as a rotation of $e_n$, let $Q_1$ and $Q_2$ denote two matrices in $\mathrm{O}(n)$ so that $\om_j = Q_j e_n$ for $j\in \{1, 2\}$. 
As justified in subsection \ref{sec: stationary states}, we construct the stationary states 
\[
E_\lambda f_\ve^{Q_j e_n} + v_j , \quad j\in \{1, 2\}
\]
belonging in the $B^*(\R^n)$-space, where $v_j = (\Id - P_\lambda \circ V_j)^{-1}[P_\lambda(V_j u_j)]$; here $\lambda > \max (|\kappa|/2, \lambda_{V_j})$ and $0<\ve\leq 1$. 
Moreover, we also have that for every $\eta > 0$ there is a dimensional constant $C(n)=C>0$ such that 
\[
\|v_j\|_{B^*(\R^n)} \leq \frac{C\lambda_{V_j}}{\lambda} \|E_\lambda f^{\om_j}_\ve \|_{B^*(\R^n)}
\]
for all $\lambda \geq \max(|\kappa|/2, \lambda_{V_j}+ \eta)$ and $\ve>0$. In what follows we consider $\ve$ to be sufficiently small. 
Recall now that the functions 
\[
\psi_j(t, x)=e^{-i\lambda^2t}(E_\lambda f^{\om_j}_\ve (x) + v_j(x)), \quad j\in\{1,2\},
\]
are time-harmonic solutions of the Schr\"odinger equation with potential $V_j$, $j\in\{1,2\}$, and belong to the space $C([0, T]; B^*(\R^n)) $. 

The following observation is necessary for the proof of \eqref{eq: stability2_est_t_independent}.
%
%%%%%%LEMMA%%%%%%LEMMA%%%%%%%%%%
\begin{lemma}\label{lem: V_in_B_space}
If $V\in L^1(\R^n)\cap L^\infty(\R^n)$ with $\vvvert V \vvvert < \infty$, then $V\in B(\R^n)$. 
\end{lemma}
%
%%%%%%%PROOF%%%%%%%
\begin{proof}
Using Cauchy-Schwarz, we estimate as follows 
\begin{align*}
\|V\|_{B(\R^n)} &= \sum_{j\in\mathbb N_0} 2^{j/2} \|V\|_{L^2(D_j)} \leq \sum_{j\in\mathbb N_0}2^{j/2} \|V\|_{L^\infty}^{1/2} \|V\|_{L^1(D_j)}^{1/2}
\leq \Big( \sum_{j\in\mathbb N_0} 2^j \|V\|_{L^\infty} \Big)^{1/2} \Big (  \sum_{j\in\mathbb N_0} \|V\|_{L^1(D_j)} \Big)^{1/2} \\
&= \vvvert V \vvvert^{1/2} \|V\|_{L^1(\R^n)}^{1/2} < \infty. 
\end{align*}
\end{proof}
%%%%%%%%
%
Obviously, Lemma \ref{lem: V_in_B_space} also holds under the stronger assumption $\| V \|_\delta < \infty$ and in particular we have that $V\in L^2_\delta(\R^n)$.

\vspace{0.2cm}

%
%%%%%%%%%PROOF_MAIN_THEOREM%%%%%%%%%%%%%%%%%%%
%
\noindent {\bf Proof of Theorem \ref{th: stability_t_independent}:}
We begin with the proof of the first stability estimate \eqref{eq: stability1_est_t_independent}.
We can write
\begin{align}
T\int_{\R^n} (V_1-V_2) E_\lambda f^{\om_2}_\ve E_\lambda f^{\om_2}_\ve &= 
\int_0^T \int_{\R^n} (V_1-V_2) (E_\lambda f^{\om_1}_\ve + v_1 - v_1) (E_\lambda f^{\om_2}_\ve + v_2 - v_2) \notag \\
&= \int_\Sigma (V_1-V_2) \psi_1 \ovl{\psi_2} 
- T\int_{\R^n} (V_1-V_2)[ E_\lambda f^{\om_1}_\ve v_2 + E_\lambda f^{\om_2}_\ve v_1 + v_1 v_2] \label{eq: built_eq1}.
\end{align}
%%%%%
%
Set $F:= V_1 - V_2$ and notice that 
\begin{align}
\frac{1}{(2\pi)^{n/2}} &\int_{\R^n} F E_\lambda f^{\om_1}_\ve E_\lambda f^{\om_2}_\ve =
 \int_{\Sph^{n-1}\times\Sph^{n-1}} \wh{F}(\lambda(\theta+\om)) f^{\om_1}_\ve (\theta) f^{\om_2}_\ve(\om)\, \dd\mu(\theta, \om) \notag\\
&= \wh{F}(\kappa) \|f^{\om_1}_\ve\|_{L^1(\Sph^{n-1})}\|f^{\om_2}_\ve\|_{L^1(\Sph^{n-1})} 
+ \int_{\Sph^{n-1}\times\Sph^{n-1}}\big[ \wh{F}(\lambda(\theta+\om)) - \wh{F}(\kappa)\big] f^{\om_1}_\ve (\theta) f^{\om_2}_\ve(\om)
\, \dd\mu(\theta, \om) \label{eq: built_eq2}
\end{align}
%%%%%%%
%
where $\mu$ denotes the product measure $\sigma\times \sigma$. As in \cite{zbMATH08122191}, we can bound 
\begin{equation}\label{eq: mod_of_continuity_bound}
 \int_{\Sph^{n-1}\times\Sph^{n-1}}\big[ \wh{F}(\lambda(\theta+\om)) - \wh{F}(\kappa)\big] f^{\om_1}_\ve (\theta) f^{\om_2}_\ve(\om)
\, \dd\mu(\theta, \om)  
\leq \gamma(\lambda \ve),
\end{equation}
where
\[
\gamma(\rho):= \frac{1}{(2\pi)^{n/2}} \int_{\R^n} \sup_{|\xi|\leq\rho} |e^{i\xi\cdot x}-1||F(x)|\, \dd x, \quad \text{with} \,\, \rho\in(0, \infty),
\]
is a modulus of continuity. Thus, using \eqref{eq: built_eq2} and \eqref{eq: mod_of_continuity_bound}, from \eqref{eq: built_eq1} we get that 
\begin{align}
|\wh{F}(\kappa)| \|f^{\om_1}_\ve\|_{L^1(\Sph^{n-1})}\|f^{\om_2}_\ve\|_{L^1(\Sph^{n-1})} &\lec   
\gamma(\lambda \ve) \|f^{\om_1}_\ve\|_{L^1(\Sph^{n-1})}\|f^{\om_2}_\ve\|_{L^1(\Sph^{n-1})}  \notag \\
&+ \Big| \int_\Sigma F \psi_1 \ovl{\psi_2} \Big|  
+ T\Big| \int_{\R^n} F [E_\lambda f^{\om_1}_\ve v_2 + E_\lambda f^{\om_2}_\ve v_1 + v_1 v_2] \Big| 
\label{eq: FT_stability_estimate}
\end{align}
where the constant depends on the dimension $n$. 

Notice that, under the integrability condition \eqref{eq: int_cond_V} for $V_j$, $j\in\{1, 2\}$, 
we have that $\gamma(\lambda \ve) \lec \lambda \ve$. 
Indeed, there holds 
\begin{equation}\label{eq: complex_exp_bound}
|e^{i \xi \cdot x} - 1|=\Big|\int_0^{\xi\cdot x} (e^{is})' \, \dd s \Big| \leq \int_0^{\xi\cdot x} 1\, \dd s \leq |\xi| |x|.
\end{equation}
Thus, $\sup_{|\xi|\leq \rho}|e^{i\xi\cdot x} - 1| \leq \rho |x|$, which implies that
\[
\gamma(\rho) \leq \frac{\rho}{(2\pi)^{n/2}} \int_{\R^n} |x||F(x)| \, \dd x .
\]
Thus, we readily have that $\gamma(\lambda \ve)\lec_n \lambda \ve$. 

In the view of Lemma \ref{lem: bound_int_t-harmonic_by_distance} and \eqref{eq: bstarl2delta}, we infer that
\begin{align*}
\Big| \int_\Sigma (V_1-V_2) \psi_1 \ovl{\psi_2} \Big| &\leq
 T\dist_\delta(\cal{U}^1_T, \cal{U}^2_T) \|\psi_1\|_{C_T L^2_{-\delta}(\R^n)} \|\psi_2\|_{C_T L^2_{-\delta}(\R^n)}\\
&\lec T\dist_\delta(\cal{U}^1_T, \cal{U}^2_T) \|\psi_1\|_{C_TB^*(\R^n)} \|\psi_2\|_{C_T B^*(\R^n)} 
\lec \dist_\delta(\cal{U}^1_T, \cal{U}^2_T) \Big(1+\frac{1}{\lambda}\Big)^2 \frac{1}{\lambda^{n-1}} \frac{1}{\ve^{n-1}}.
\end{align*}
%%%%%%%
%
Estimating in the same way as in the proof of \cite[Theorem 1]{zbMATH08122191}, we get that 
\[
\Big| \int_{\R^n} F [E_\lambda f^{\om_1}_\ve v_2 + E_\lambda f^{\om_2}_\ve v_1 + v_1 v_2] \Big| 
\lec \vvvert F \vvvert \Big( \frac{1}{\lambda^n} + \frac{1}{\lambda^{n+1}} \Big) \frac{1}{\ve^{n-1}}.
\]
Combining the above estimates, we infer that 
\begin{align*}
|\wh{F}(\kappa)| &\lec 
\lambda \ve + \dist_\delta(\cal{U}^1_T, \cal{U}^2_T)  \frac{1}{\lambda^{n-1}} \frac{1}{\ve^{n-1}} \frac{1}{\|f^{\om_1}_\ve\|_{L^1}}\frac{1}{\|f^{\om_1}_\ve\|_{L^1}}
+ \vvvert F \vvvert \Big( \frac{1}{\lambda^n} + \frac{1}{\lambda^{n+1}} \Big) \frac{1}{\ve^{n-1}} \frac{1}{\|f^{\om_1}_\ve\|_{L^1}}\frac{1}{\|f^{\om_1}_\ve\|_{L^1}} \\
&\lec \dist_\delta(\cal{U}^1_T, \cal{U}^2_T)  \frac{1}{\lambda^{n-1}} \frac{1}{\ve^{n-1}} + \lambda \ve + \frac{1}{\lambda^n}\frac{1}{\ve^{n-1}}.
\end{align*}
%%%%%%%
%
In the last inequality we used the fact that since $0<\ve\ll1$ is small enough, the norms  
${\|f^{\om_1}_\ve\|_{L^1}} , \, {\|f^{\om_1}_\ve\|_{L^1}} \approx C$, where $C>0$ constant. 
Recall that $F=V_1-V_2$ and by splitting in high and low frequencies, we estimate as follows 
%
%%%%%%%
\begin{align*}
\|V_1-V_2\|_{H^{-1}(\R^n)}^2 &= \Big( \int_{|\kappa|\leq \rho} + \int_{|\kappa|>\rho} \Big) \frac{|\wh{(V_1-V_2)}(\kappa)|^2}{1+|\kappa|^2}\, \dd\kappa \\
&\leq \rho^n \Big( \dist_\delta(\cal{U}^1_T, \cal{U}^2_T)  \frac{1}{\lambda^{n-1}} \frac{1}{\ve^{n-1}} + \lambda \ve + \frac{1}{\lambda^n}\frac{1}{\ve^{n-1}}\Big)^2 + \frac{1}{\rho^2} \| V_1-V_2 \|_{L^2(\R^n)}^2.
\end{align*}
%%%%%%%
%
Choose $\ve=\lambda^{-1} \lambda^{-1/n}$. Then, 
\[
\|V_1-V_2\|_{H^{-1}(\R^n)}^2 \lec \rho^n  \Big( \dist_\delta(\cal{U}^1_T, \cal{U}^2_T)  \frac{\lambda}{\lambda^{1/n}} + \frac{2}{\lambda^{1/n}}\Big)^2 + \frac{1}{\rho^2} 
\lec \rho^n \dist_\delta(\cal{U}^1_T, \cal{U}^2_T)^2  \frac{\lambda^2}{\lambda^{2/n}} + \frac{\rho^n}{\lambda^{2/n}} + \frac{1}{\rho^2}.
\]
Making the last two terms of equal size by choosing $\rho=\lambda^{\frac{2}{n(n+2)}}$, we get that 
\[
\|V_1-V_2\|_{H^{-1}(\R^n)}^2 \lec \dist_\delta(\cal{U}^1_T, \cal{U}^2_T)^2 \frac{\lambda^2}{\lambda^{\frac{4}{n(n+2)}}} 
+ \frac{2}{\lambda^{\frac{4}{n(n+2)}}}. 
\]
Finally, by choosing $\lambda= \dist_\delta(\cal{U}^1_T, \cal{U}^2_T)^{-1}$, we can easily infer that the estimate
\eqref{eq: stability1_est_t_independent} holds true.

\vspace{0.4cm}

We turn now our attention to the refined estimate  \eqref{eq: stability2_est_t_independent}. We start again with the estimate \eqref{eq: FT_stability_estimate} and recall that under the condition \eqref{eq: int_cond_V} for $V_j$, $j\in\{1, 2\}$, it holds that $\gamma(\lambda \ve) \lec \lambda \ve$. 
As before, by Lemma \ref{lem: bound_int_t-harmonic_by_distance}, we also have that 
\[
\Big| \int_\Sigma (V_1 - V_2)\psi_1 \ovl{\psi_2}\Big| \leq T \dist_\delta(\cal{U}^1_T, \cal{U}^2_T) \| \psi_1\|_{C_TL^2_{-\delta}} \|\psi_2\|_{C_TL^2_{-\delta}}.
\]
The difference and thus the improvement here, comes from the fact that we do not measure the whole solution $\psi_j$ in the $C_TB^*$-norm, but we are taking advantage the $L^\infty$-norm of the leading terms instead. 
In particular, we have that 
\[
\| \psi_j \|_{C_TL^2_{-\delta}} \leq \| E_\lambda f^{\om_j}_\ve \|_{L^2_{-\delta}(\R^n)} + \| v_j \|_{B^*(\R^n)}, \quad \textup{for} \,\, j\in\{1, 2\}. 
\]
Observe that for $\delta > n/2$ there holds 
\[
\int_{\R^n} \frac{1}{(1+|x|^2)^{\delta}} \, \dd x<\infty. 
\]
So, with the above integrability for the weight, we can estimate the leading term as follows: 
\begin{equation}\label{eq: Herglotz_bound}
\|E_\lambda f^{\om_j}_\ve \|_{L^2_{-\delta}(\R^n)} \lec \|E_\lambda f^{\om_j}_\ve\|_{L^\infty(\R^n)} \leq \|f^{\om_j}_\ve \|_{L^1(\Sph^{n-1})},
\quad \text{for} \,\, j\in\{1, 2\} \,\, \text{and} \,\, \delta>n/2.
\end{equation}
In order to bound the perturbations, first recall that $V_j\in B(\R^n)$ by Lemma \ref{lem: V_in_B_space}. Then, by taking advantage the boundedness of the operator $(\Id-P_\lambda \circ V)^{-1}$ in the $B^*(\R^n)$-norm along with the estimate \eqref{eq: Plambda_bound}, we readily have that 
\begin{equation}\label{eq: perturbation_Bstar_bound}
\| v_j \|_{B^*(\R^n)}\leq \| P_\lambda (V_j E_\lambda f^{\om_j}_\ve) \|_{B^*(\R^n)} \leq \frac{1}{\lambda} \| V_j \|_{B(\R^n)} \|E_\lambda f^{\om_j}_\ve \|_{L^\infty(\R^n)}
\leq \frac{1}{\lambda} \| V_j \|_{B(\R^n)} \|f^{\om_j}_\ve \|_{L^1(\Sph^{n-1})}, 
\end{equation}
for  $j\in\{1, 2\}$. For the boundedness of operator $(\Id-P_\lambda \circ V)^{-1}$ recall the brief construction of the solutions in subsection \ref{sec: stationary states} or see \cite[Corollary 2.4]{zbMATH08122191}.

As a result, using the estimates \eqref{eq: perturbation_Bstar_bound} and \eqref{eq: Herglotz_bound}, we get
\begin{align*}
&\Big| \int_\Sigma (V_1- V_2)\psi_1 \ovl{\psi_2} \Big|  \\
&\lec \dist_\delta(\cal{U}^1_T, \cal{U}^2_T) \Big( \|f^{\om_1}_\ve\|_{L^1(\Sph^{n-1})} + \frac{1}{\lambda} \|V_1\|_{B(\R^n)} \|f^{\om_1}_\ve\|_{L^1(\Sph^{n-1})} \Big)  
\Big( \|f^{\om_2}_\ve\|_{L^1(\Sph^{n-1})} + \frac{1}{\lambda} \|V_2\|_{B(\R^n)}\|f^{\om_2}_\ve\|_{L^1(\Sph^{n-1})} \Big) \\
&= \dist_\delta(\cal{U}^1_T, \cal{U}^2_T) \|f^{\om_1}_\ve\|_{L^1(\Sph^{n-1})} \|f^{\om_2}_\ve\|_{L^1(\Sph^{n-1})}
+ \dist_\delta(\cal{U}^1_T, \cal{U}^2_T) \frac{1}{\lambda^2}\|V_1\|_{B(\R^n)}\|V_2\|_{B(\R^n)} \|f^{\om_1}_\ve\|_{L^1(\Sph^{n-1})} \|f^{\om_2}_\ve\|_{L^1(\Sph^{n-1})}\\
&+ \dist_\delta(\cal{U}^1_T, \cal{U}^2_T)  \|f^{\om_1}_\ve\|_{L^1(\Sph^{n-1})}\|f^{\om_2}_\ve\|_{L^1(\Sph^{n-1})} \big(\|V_1\|_{B(\R^n)}+\|V_2\|_{B(\R^n)}\big) \frac{1}{\lambda} \\
&\lec  \dist_\delta(\cal{U}^1_T, \cal{U}^2_T) \|f^{\om_1}_\ve\|_{L^1(\Sph^{n-1})} \|f^{\om_2}_\ve\|_{L^1(\Sph^{n-1})},
\end{align*}
where in the last inequality we used the fact that $1\ll \lambda < \lambda^2$. 
Moreover, using \eqref{eq: perturbation_Bstar_bound}, Lemma \ref{lem: V_in_B_space} and \cite[Lemma 2.3]{zbMATH08122191}, we obtain
\begin{align*}
&\Big| \int_{\R^n}(V_1-V_2)[E_\lambda f^{\om_1}_\ve v_2 + E_\lambda f^{\om_2}_\ve v_1 + v_1 v_2] \Big| \\
&\leq \|V_1-V_2\|_{B(\R^n)}\big( \|E_\lambda f^{\om_1}_\ve\|_{L^\infty(\R^n)} \|v_2\|_{B^*(\R^n)} 
+ \|E_\lambda f^{\om_2}_\ve\|_{L^\infty(\R^n)} \|v_1\|_{B^*(\R^n)}\big) \\
&+ \vvvert V_1-V_2\vvvert \|v_1\|_{B^*(\R^n)} \|v_2\|_{B^*(\R^n)} \\
&\leq \|V_1-V_2\|_{B(\R^n)} \big( \|f^{\om_1}_\ve\|_{L^1(\Sph^{n-1})} \|V_2\|_{B(\R^n)} \frac{1}{\lambda} \|f^{\om_2}_\ve\|_{L^1(\Sph^{n-1})} +  \|f^{\om_2}_\ve\|_{L^1(\Sph^{n-1})} \|V_1\|_{B(\R^n)} \frac{1}{\lambda} \|f^{\om_1}_\ve\|_{L^1(\Sph^{n-1})} \big) \\
&+\vvvert V_1-V_2\vvvert \frac{1}{\lambda^2} \|V_1\|_{B(\R^n)}\|V_2\|_{B(\R^n)} \|f^{\om_1}_\ve\|_{L^1(\Sph^{n-1})} \|f^{\om_2}_\ve\|_{L^1(\Sph^{n-1})} \\
&\lec \frac{1}{\lambda}  \|f^{\om_1}_\ve\|_{L^1(\Sph^{n-1})} \|f^{\om_2}_\ve\|_{L^1(\Sph^{n-1})}.
\end{align*}
As a result, combining the above estimates, inequality \eqref{eq: FT_stability_estimate} implies that 
\[
\wh{(V_1-V_2)}(\kappa) \lec \dist_\delta(\cal{U}_T^1, \cal{U}_T^2) + \lambda \ve + \frac{1}{\lambda}.
\]
Splitting in high and low frequencies, we estimate 
\begin{align*}
\|V_1-V_2\|_{H^{-1}(\R^n)}^2 &= \Big( \int_{|\kappa|\leq \rho} + \int_{|\kappa|>\rho} \Big) \frac{|\wh{(V_1-V_2)}(\kappa)|^2}{1+|\kappa|^2}\, \dd\kappa \\
&\leq \rho^n \Big( \dist_\delta(\cal{U}^1_T, \cal{U}^2_T)  + \lambda \ve + \frac{1}{\lambda} \Big)^2 + \frac{1}{\rho^2} \| V_1-V_2 \|_{L^2(\R^n)}^2.
\end{align*}
Making the obvious choice $\ve:= \lambda^{-2}$, we get that 
\[
\|V_1-V_2\|_{H^{-1}(\R^n)}^2 \lec \rho^n \dist_\delta(\cal{U}^1_T, \cal{U}^2_T)^2 + \frac{\rho^n}{\lambda^2} + \frac{1}{\rho^2}.
\]
By setting $\rho:= \lambda^\frac{2}{n+2}$, the last two terms are becoming of equal size and we get that 
\[
\|V_1-V_2\|_{H^{-1}(\R^n)}^2 \lec \lambda^\frac{2n}{n+2} \dist_\delta(\cal{U}^1_T, \cal{U}^2_T)^2 + \frac{1}{\lambda^\frac{4}{n+2}}.
\]
Finally, by choosing $\lambda:= \dist_\delta(\cal{U}^1_T, \cal{U}^2_T)^{-\frac{n+2}{2n}}$, we conclude that 
\[
\|V_1-V_2\|_{H^{-1}(\R^n)}^2 \lec  \dist_\delta(\cal{U}^1_T, \cal{U}^2_T) +  \dist_\delta(\cal{U}^1_T, \cal{U}^2_T)^\frac{2}{n} 
\leq 2  \dist_\delta(\cal{U}^1_T, \cal{U}^2_T)^\frac{2}{n}, 
\]
where in the last inequality we used the fact that we consider distances such that $ 0< \dist_\delta(\cal{U}^1_T, \cal{U}^2_T) < 1$ and $\frac{2}{n} \leq 1$. 
Thus, we get the stability estimate \eqref{eq: stability2_est_t_independent} and the proof of the theorem is now complete. 
\qed

%%%%%%%SUBSECTION%%%%%%%%%%%
%
\subsection{Time-dependent potentials}\label{ssec: stability_time_dependent}

In this subsection we provide the proof of the Theorem \ref{th: stability_t_dependent}. As we have already mentioned, in order to control the CGO solutions, the relative weight that we use here is $w_{\delta, \ve}(x)=e^{-2|x|^{1+\ve}} \langle x \rangle^{-2 \delta}$. 
%
%It might be also useful to recall that the weighted space $L^2_{w_{\delta, \ve}}(\Sigma)$ can be identified with $L^2_TL^2_{-\delta, -\ve}$. 
%
We start with the following lemma:
%
%%%%%%%LEMMA%%%%%%%%%
\begin{lemma}\label{lem: stability_tDep_1}
	Let $T>0$,  $\delta>1/2$ and $\ve>0$. For $\nu\in\R^n$ with $|\nu|$ sufficiently large, let $\nu_1=-\nu_2=\nu$, and for $j=1,2$,
	\[
	\mathbf{u}_j=e^{\varphi_j}(u_j^\sharp+u_j^\flat),
	\]
	with 
	\[
	\varphi_j(t,x)=i|\nu|^2t+\nu_j\cdot x,
	\] 
	\[
u_j^\sharp(t,x)=[e^{it\Delta'}\check {f_j}](Q_je_1\cdot x,\ldots, Q_je_{n-1}\cdot x),
        \]
and $u_j^\flat$ as constructed in subsection \ref{sect: CGO}. Here $Q_j\in O(n)$ is a rotation such that $Q_je_n=\nu_j/|\nu_j|$.

Denote with $s=|\nu|$. There exists a global constant $C>0$ such that
\begin{equation}\label{eq: stability_tDep_1}
\left|\int_{\Sigma} (V_1-V_2)u^\sharp_1\ovl{u^\sharp_2}\right| \leq C \left(\dist_{\delta,\ve}(\cal{U}_T^1, \cal{U}_T^2)e^{2\ve s^{(1+\ve)/\ve}}+s^{-1/2}\right)\|f_1\|_{L^2(\R^{n-1})} \|f_2\|_{L^2(\R^{n-1})}.
\end{equation}
\end{lemma}
% 
%%%%%%%%PROOF%%%%%%%
\begin{proof}
	First observe that,
	\[
	\int_{\Sigma} (V_1-V_2)u^\sharp_1\ovl{u^\sharp_2}=\int_{\Sigma} (V_1-V_2) \mathbf{u}_1\ovl{\mathbf{u}_2}-\int_{\Sigma} (V_1-V_2)(u^\flat_1\ovl{u^\sharp_2}+u^\sharp_1\ovl{u^\flat_2}+u^\flat_1\ovl{u^\flat_2}).
	\]
	We can estimate directly the first right-hand side term by Lemma \ref{lem: bound_integral_CGO_by_distance} as follows 
	\[
	\left|\int_{\Sigma} (V_1-V_2)\mathbf{u}_1\ovl{\mathbf{u}_2}\right| 
	\leq 
	\dist_{\delta,\ve}(\cal{U}_T^1, \cal{U}_T^2) \|\mathbf{u}_1\|_{L^2_{w_{\delta, \ve}}(\Sigma)} \|\mathbf{u}_2\|_{L^2_{w_{\delta, \ve}}(\Sigma)}.
	\]
	By the estimate \eqref{eq: CGO_solution_in_H_norm} we have that
	\[
	\|\mathbf{u}_j\|_{L^2_{w_{\delta, \ve}}(\Sigma)} \lec e^{\ve s^{(1+\ve)/\ve}} \|f_j\|_{L^2(\R^{n-1})},
	\]
	which gives us the first part of the estimate \eqref{eq: stability_tDep_1}. The second term can be estimated by H\"older's inequality and the estimates \eqref{eq:solSharpEst2} and \eqref{eq:solFlatEst2}.
	For instance,
	\begin{equation}\label{eq: est_td_flat_sharp}
	\begin{aligned}
		\left| \int_{\Sigma} (V_1-V_2)u^\flat_1\ovl{u^\sharp_2}\right| &\leq \|\langle \cdot \rangle^{2\delta}(V_1-V_2)\|_{L^\infty(\Sigma)} \|u^\flat_1\|_{L^2_TL^2_{-\delta}} \|u^\sharp_2\|_{L^2_TL^2_{-\delta}} \\
		&\lec s^{-1/2}\|f_1\|_{L^2(\R^{n-1})}\|f_2\|_{L^2(\R^{n-1})}.
	\end{aligned}
	\end{equation}
	Recall that in $\langle \cdot \rangle$ we consider values that belong to $\R^n$. Above we have used the fact that since $V_j$ with $j\in\{1, 2\}$ have super exponential decay, $\langle \cdot \rangle^{2\delta}(V_1-V_2)$ is bounded in $L^\infty(\Sigma)$. In a similar fashion, we obtain 
	\begin{equation}\label{eq: est_td_sharp_flat}
		\left| \int_{\Sigma} (V_1-V_2)u^\sharp_1\ovl{u^\flat_2}\right|\lec s^{-1/2}\|f_1\|_{L^2(\R^{n-1})}\|f_2\|_{L^2(\R^{n-1})}
	\end{equation}
	and
	\begin{equation}\label{eq: est_td_flat_flat}
		\left| \int_{\Sigma} (V_1-V_2)u^\flat_1\ovl{u^\flat_2}\right|\lec s^{-1}\|f_1\|_{L^2(\R^{n-1})}\|f_2\|_{L^2(\R^{n-1})}
	\end{equation}
	Putting \eqref{eq: est_td_flat_sharp}, \eqref{eq: est_td_sharp_flat} and \eqref{eq: est_td_flat_flat} together gives us the second term of estimate \eqref{eq: stability_tDep_1}, and thus the lemma is proven.
\end{proof}
We remark that for the rest of this section we will denote  $s=|\nu|$.

\vspace{0.2cm}

\noindent {\bf Concentration of solutions:}
Now, we will see how the LHS of \eqref{eq: stability_tDep_1} gives us information on the Fourier transform of $\wt{V}_1-\wt{V}_2$, where again $\wt{V}_j$ denotes the trivial extension of $V_j$ to $R\times\R^n$. 
Indeed, write 
$F = \wt{V}_1 - \wt{V}_2 $ and let $u^\sharp_1$ and $u^\sharp_2$ be as in Lemma \ref{lem: stability_tDep_1}. Using Fubini we have
\[
\begin{aligned}
&\int_{\R\times\R^n} F(t,x)u^\sharp_1(t,x)\ovl{u^\sharp_2}(t,x)\dd t \dd x \\ 
&= \frac{1}{(2\pi)^{n-1}} \int_{\R^{n-1} \times \R^{n-1} \times \R \times \R^n} F(t,x)e^{ix\cdot(Q_1 \bs{\xi}-Q_2  \bs{\eta})}e^{-it(|\xi|^2-|\eta|^2)}f_1(\xi)f_2(\eta) \dd t\, \dd x\, \dd \xi \, \dd \eta \\
&= 4 \pi^2 \int_{\R^{n-1}\times\R^{n-1}} \widehat{F}(|\xi|^2-|\eta|^2, Q_2 \bs{\eta}-Q_1  \bs{\xi})f_1(\xi)f_2(\eta) \dd \xi\, \dd \eta.
\end{aligned}
\]
Here $\widehat F$ stands for the $(n+1)$-dimensional Fourier transform of $F$. Recall also that $\bs{\xi}=(\xi, 0)$ for $\xi \in \R^{n-1}$. 

First, let $(\tau_0, \sigma_0)\in \R\times\left(\R^n\setminus\{0\}\right)$. We can choose $\xi_0,\eta_0\in\R^{n-1}$ such that $(|\xi_0|^2-|\eta_0|^2,Q_2 \bs{\eta}_0 - Q_1\bs{\xi}_0)=(\tau_0,\sigma_0)$. Indeed, we make the choice
\[
\bs{\xi}_0=-\frac{1}{2}\left(1+\frac{\tau_0}{|\sigma_0|^2}\right)Q_1^T\sigma_0, \quad   \bs{\eta}_0=\frac{1}{2}\left(1-\frac{\tau_0}{|\sigma_0|^2}\right)Q_2^T\sigma_0.
\]
Note that, by definition of $u_j^\sharp$, the above integral equality holds for any $f_j\in L^2(\R^{n-1})$, with $j\in\{1, 2\}$. So, the idea is to choose $f_1$ and $f_2$ appropriately concentrated around $\xi_0$ and $\eta_0$, respectively. To do so, fix $\chi\in \mathcal C_0^\infty(\R^{n-1})$ supported in $\{x\in\R^{n-1}:|x|\leq 1\}$, such that $\chi\geq0$ and $\|\chi\|_{L^1(\R^{n-1})}=1$. Then just take an $L^1$-scaling for $\rho>0$:
\[
f_1^\rho(\xi) = \rho^{-(n-1)}\chi\left(\frac{\xi-\xi_0}{\rho}\right),\quad f_2^\rho(\eta) = \rho^{-(n-1)}\chi\left(\frac{\eta-\eta_0}{\rho}\right).
\]
It is easy to see that $\| f^\rho_1\|_{L^1(\R^{n-1})}=\| f^\rho_2\|_{L^1(\R^{n-1})}=\|\chi\|_{L^1(\R^{n-1})}=1$. Now, call 
\[
G( \xi,  \eta)=\widehat{F}(|\xi|^2-|\eta|^2, Q_2  \bs{\eta}-Q_1 \bs{\xi}),
\] 
and observe that
\begin{equation}\label{eq: concentration}
	\begin{aligned}
	\int_{\R^{n-1}\times\R^{n-1}}G(\xi, \eta)f_1^\rho (\xi)f_2^\rho(\eta) \dd \xi\dd \eta &= G(\xi_0,\eta_0)\int_{\R^{n-1}\times\R^{n-1}}f_1^\rho (\xi)f_2^\rho(\eta) \dd \xi \dd \eta \\
	&+ \int_{\R^{n-1}\times\R^{n-1}}\left[G(\xi,\eta)-G(\xi_0,\eta_0)\right] f_1^\rho (\xi)f_2^\rho(\eta) \dd \xi \dd \eta.
	\end{aligned}
\end{equation}
Clearly, the first summand in the RHS of \eqref{eq: concentration} is equal to $G(\xi_0,\eta_0)=\widehat F(\tau_0,\sigma_0)$. 

The second summand can be bounded by the quantity
\[
\gamma(\rho)=\sup_{\substack{|\xi-\xi_0|<\rho\\|\eta-\eta_0|<\rho}}\left|G(\xi,\eta)-G(\xi_0,\eta_0)\right|\leq\sup_{\substack{|\tau-\tau_0|<2\rho^2\\|\sigma-\sigma_0|<2\rho}}\left|\widehat F(\tau,\sigma)-\widehat F(\tau_0,\sigma_0)\right|.
\]
Therefore, we get that
\begin{align*}
\gamma(\rho)&\leq\sup_{\substack{|\tau-\tau_0|<2\rho^2\\|\sigma-\sigma_0|<2\rho}}\left|\int_{\R\times\R^n}F(t,x)\left[e^{-i(t\tau+x\cdot\sigma)}-e^{-i(t\tau_0+x\cdot\sigma_0)}\right]\dd t\,\dd x\right|\\
&\leq \left(\int_{\R\times\R^n}(1+|t|+|x|)\left|F(t,x)\right|\dd t\,\dd x\right)
\left(\sup_{\substack{(t,x)\in \R\times \R^n\\|\tau - \tau_0|<2\rho^2\\|\sigma-\sigma_0|<2\rho}}\frac{\left|e^{-i(t(\tau-\tau_0)
+x\cdot (\sigma-\sigma_0))}-1\right|}{1+|t|+|x|}\right).
\end{align*}
On the one hand, since $F=0$ if $(t, x) \notin \Sigma$, we have that
\[
\int_{\R\times\R^n}(1+|t|+|x|)\left|F(t,x)\right|\dd t\,\dd x\leq T \|F\|_{L^1(\R \times \R^n)} + \|e^{|x|}F\|_{L^\infty((0, T)\times\R^n)}\int_{\R^n}e^{-|x|}(1+|x|)\dd x<\infty,
\]
as long as $F$ has exponential decay in $L^\infty$, which is true for $F=\wt{V}_2-\wt{V}_1$. So, a computation analogous to \eqref{eq: complex_exp_bound}, shows that
\[
\frac{\left|e^{-i(t(\tau-\tau_0)+x\cdot (\sigma-\sigma_0))}-1\right|}{1+|t|+|x|} 
\leq \frac{|t||\tau-\tau_0|+|x||\sigma-\sigma_0|}{1+|t|+|x|} 
\leq |\tau-\tau_0|+ |\sigma-\sigma_0|.
\]
Hence, 
\begin{equation}\label{eq: modulus_time_dep}
\gamma(\rho)\lesssim 
\sup_{\substack{|\tau-\tau_0|<2\rho^2\\|\sigma-\sigma_0|<2\rho}} \big(|\tau-\tau_0|+ |\sigma-\sigma_0| \big)<2\rho^2+2\rho 
\lesssim \rho
\end{equation}
whenever $\rho< 1$. Thus \eqref{eq: concentration} and \eqref{eq: modulus_time_dep} yield
\begin{equation}\label{eq: FT_time_dep}
	\left|\widehat F(\tau_0,\sigma_0)\right|\lesssim\left|\int_{ \R \times\R^n} Fu^\sharp_1\ovl{u^\sharp_2}\right| + \rho
	=\left| \int_{(0, T) \times \R^n} (V_1-V_2) u^\sharp_1 \ovl{u^\sharp_2}\right| + \rho,
\end{equation}
where the implicit constant depends on $T$ and $F$. Of course here and in the forthcoming proof of the Theorem \ref{th: stability_t_dependent}, $u^\sharp_j$ depends on $\rho$ as
\[
u_j^\sharp(t,x)=[e^{it\Delta'}\check {f_j^\rho}](Q_je_1\cdot x,\ldots, Q_je_{n-1}\cdot x) , \quad \text{for} \, \, j\in \{1, 2\}.
\]

%%%%%%%%%%%%%%%%%%
%
%%%%%%%%%%PROOF_STABILITY_TIME%%%%%%%%%%%
\noindent{\bf Proof of Theorem \ref{th: stability_t_dependent}:} 
For any $(\tau,\sigma)\in \R\times\left(\R^n\setminus\{0\}\right)$ and $0<\rho<1$, choose $f_1^\rho$ and $f_2^\rho$ concentrating around points $\xi$ and $\eta$ such that \[
\xi=-\frac{1}{2}\left(1+\frac{\tau}{|\sigma|^2}\right)Q_1^T\sigma, \quad   \eta=\frac{1}{2}\left(1-\frac{\tau}{|\sigma|^2}\right)Q_2^T\sigma,
\] 
as above, and choose solutions as in Lemma \ref{lem: stability_tDep_1} with $f_j= f_j^\rho$, $j=1,2$. 
Recall that $s=|\nu|$ and $F=\wt{V}_1-\wt{V}_2$, where $\wt{V}_j$ denotes the extension of $V_j$ to $\R\times \R^n$. From \eqref{eq: FT_time_dep} and \eqref{eq: stability_tDep_1} we infer that
\[
	\left|\widehat F(\tau,\sigma)\right|\lesssim \left(\dist_{\delta,\ve}(\cal{U}_T^1, \cal{U}_T^2)e^{2\ve s^{(1+\ve)/\ve}}+s^{-1/2}\right)\|f^\rho_1\|_{L^2(\R^{n-1})} \|f^\rho_2\|_{L^2(\R^{n-1})} + \rho.
\]
Note that the concentration in $L^1$ as $\rho\to 0$ will make the $L^2$ norm of $f^\rho_j$ grow: 
\[
\|f_j^\rho\|_{L^2(\R^{n-1})}=\rho^{-(n-1)/2},
\]
so that
\[
	\left|\widehat F(\tau,\sigma)\right|\lesssim \left(\dist_{\delta,\ve}(\cal{U}_T^1, \cal{U}_T^2)e^{2\ve s^{(1+\ve)/\ve}}+s^{-1/2}\right)\rho^{-(n-1)} + \rho.
\]
From now on, we will choose the approximation rate $\rho$ as a function of $s$, in particular $\rho=s^{-1/2n}$, which yields
\begin{equation}\label{eq: FT_stability_time_dep}
	\left|\widehat F(\tau,\sigma)\right|\lesssim \dist_{\delta,\ve}(\cal{U}_T^1, \cal{U}_T^2)s^{(n-1)/2n}e^{2\ve s^{(1+\ve)/\ve}}+s^{-1/2n}.
\end{equation}
Now, choose some $R>0$ and denote $\xi=(\tau,\sigma)$ to improve readability. We can estimate the $H^{-1}$ norm of $V_1-V_2$ as follows:
\begin{equation}\label{eq: H_1_time_dep}
\begin{aligned}
	\|V_1-V_2\|^2_{H^{-1}(\R^{n+1})}&=\int_{\R^{n+1}}\frac{|\widehat F(\xi)|^2}{1+|\xi|^2}\,\dd \xi=\left(\int_{|\xi|\leq R}\dd \xi + \int_{|\xi|> R}\dd \xi \right) \frac{|\widehat F(\xi)|^2}{1+|\xi|^2} \\
	&\lesssim \left(\dist_{\delta,\ve}(\cal{U}_T^1, \cal{U}_T^2)^2 s^{(n-1)/n}e^{4\ve s^{(1+\ve)/\ve}}+s^{-1/n} \right) R^{n+1}+R^{-2}\|V_1-V_2\|^2_{L^2(\R^{n+1})}.
\end{aligned}
\end{equation}
Here we have used \eqref{eq: FT_stability_time_dep} to bound the first term, noting that the estimate \eqref{eq: FT_stability_time_dep} is in principle true for all $\xi\in\R^{n+1}$ except for a set of measure $0$. Moreover, the super-exponential decay of $V_j$ implies that 
\[
	\|V_1-V_2\|_{L^2(\R^{n+1})}^2\leq T \|e^{|x|}(V_1-V_2)\|_{L^\infty(\R\times\R^n)}^2 \int_{\R^n} e^{-|x|}\dd x <\infty.
\]
Using this and choosing $R=s^{1/(n(n+3))}$ so that $R^{-2}=R^{n+1}s^{-1/n}$, turn \eqref{eq: H_1_time_dep} into the estimate
\[
\|V_1-V_2\|^2_{H^{-1}(\R^{n+1})}\lesssim \dist_{\delta,\ve}(\cal{U}_T^1, \cal{U}_T^2)^2 s^{1-2/(n(n+3))}e^{4\ve s^{(1+\ve)/\ve}}+s^{-2/(n(n+3))}.
\]
So, for $s$ big enough so that $s^{1-2/(n(n+3))}\leq e^{\ve s^{(1+\ve)/\ve}}$ it holds that 
\[
\|V_1-V_2\|^2_{H^{-1}(\R^{n+1})}\lesssim \dist_{\delta,\ve}(\cal{U}_T^1, \cal{U}_T^2)^2 e^{5\ve s^{(1+\ve)/\ve}}+s^{-2/(n(n+3))}.
\]
Then, if we choose $s=\left((5\ve)^{-1}\big|\log \big( \dist_{\delta,\ve}(\cal{U}_T^1, \cal{U}_T^2) \big) \big|\right)^{\ve/(1+\ve)}$ for $\dist_{\delta,\ve}(\cal{U}_T^1, \cal{U}_T^2)$ sufficiently small, we arrive to the estimate
\[
\|V_1-V_2\|^2_{H^{-1}(\R^{n+1})}\lesssim \dist_{\delta,\ve}(\cal{U}_T^1, \cal{U}_T^2)
+\left((5\ve)^{-1}\big|\log \big(\dist_{\delta,\ve}(\cal{U}_T^1, \cal{U}_T^2)\big)\big|\right)^{-2\ve/((n(n+3))(1+\ve))}.
\]
The statement of Theorem \ref{th: stability_t_dependent} follows from here, since we care for distances  
$0<\dist_{\delta,\ve}(\cal{U}_T^1, \cal{U}_T^2)<1$, in which case 
$\dist_{\delta,\ve}(\cal{U}_T^1, \cal{U}_T^2) \leq 
\left((5\ve)^{-1}\big|\log \big(\dist_{\delta,\ve}(\cal{U}_T^1, \cal{U}_T^2)\big)\big|\right)^{-2\ve/((n(n+3))(1+\ve))}$.
\qed

\begin{remark}
	It might be interesting to note that the choice of $s=\left((5\ve)^{-1}\big|\log\big(\dist_{\delta,\ve}(\cal{U}_T^1, \cal{U}_T^2)\big) \big|\right)^{\ve/(1+\ve)}$ forces $\dist_{\delta,\ve}(\cal{U}_T^1, \cal{U}_T^2)$ to be small enough so that $s$ is big enough for $s^{1-2/(n(n+3))}\leq e^{4\ve s^{(1+\ve)/\ve}}$ to hold.

	As we already noted in Remark \ref{remark1}, the limit $\ve\to 0$ makes the stability estimate \eqref{eq: stability1_est_t_dependent} worse. However, it allows us to admit bigger values of $\dist_{\delta,\ve}(\cal{U}_T^1, \cal{U}_T^2)$ for which the estimate is true, since the exponent $\ve s^{(1+\ve)/\ve}$ grows as $\ve\to 0$.
\end{remark}

%%%%%%%%%%%%%%%%%%%%%%%%%%%%%% ACKNOWLEDGMENTS ACKNOWLEDGMENTS ACKNOWLEDGMENTS
\sloppy
\begin{acknowledgements}
The authors warmly thank Pedro Caro for sharing his insights on this problem and for the useful conversations. 

\vspace{0.2cm}

\noindent  T. Zacharopoulos is supported by the grant 10.46540/3120-00003B from Independent Research Fund Denmark. 
The authors are partially supported by grant PID2024-156267NB-I00 funded by MICIU/AEI/10.13039/501100011033 and cofunded by the European Union.
\end{acknowledgements}
%%%%%%%%%%%%%%%%%%%%%%%%%%%%%% ACKNOWLEDGMENTS ACKNOWLEDGMENTS ACKNOWLEDGMENTS

%%%%%%%%%%%%%%%%%%%%%%%%%%%%%% BIBLIOGRAPHY BIBLIOGRAPHY BIBLIOGRAPHY
\bibliography{references}{}

\begin{thebibliography}{10}

\bibitem{Agmon1975}
Shmuel Agmon.
\newblock Spectral properties of schrödinger operators and scattering theory.
\newblock {\em Annali della Scuola Normale Superiore di Pisa - Classe di Scienze}, 2(2):151--218, 1975.

\bibitem{Alessandrini01011988}
Giovanni Alessandrini.
\newblock Stable determination of conductivity by boundary measurements.
\newblock {\em Applicable Analysis}, 27(1-3):153--172, 1988.

\bibitem{Alessandrini2005}
Giovanni Alessandrini and Sergio Vessella.
\newblock Lipschitz stability for the inverse conductivity problem.
\newblock {\em Advances in Applied Mathematics}, 35(2):207–241, August 2005.

\bibitem{Baudouin2002}
Lucie Baudouin and Jean-Pierre Puel.
\newblock Uniqueness and stability in an inverse problem for the schr dinger equation.
\newblock {\em Inverse Problems}, 18(6):1537–1554, October 2002.

\bibitem{Bellassoued2017}
Mourad Bellassoued.
\newblock Stable determination of coefficients in the dynamical schr\"{o}dinger equation in a magnetic field.
\newblock {\em Inverse Problems}, 33(5):055009, March 2017.

\bibitem{Bellassoued2010}
Mourad Bellassoued and David Dos~Santos Ferreira.
\newblock Stable determination of coefficients in the dynamical anisotropic schr\"{o}dinger equation from the dirichlet-to-neumann map.
\newblock {\em Inverse Problems}, 26(12):125010, November 2010.

\bibitem{BenAcha2017}
Ibtissem Ben~Aïcha.
\newblock Stability estimate for an inverse problem for the schr\"{o}dinger equation in a magnetic field with time-dependent coefficient.
\newblock {\em Journal of Mathematical Physics}, 58(7), 2017.

\bibitem{zbMATH08122191}
Manuel Ca{\~n}izares, Pedro Caro, Ioannis Parissis, and Thanasis Zacharopoulos.
\newblock The initial-to-final-state inverse problem with time-independent potentials.
\newblock {\em Inverse Probl.}, 41(11):19, 2025.
\newblock Id/No 115001.

\bibitem{arXiv:2602.12122}
Manuel Ca{\~n}izares, Pedro Caro, Ioannis Parissis, and Thanasis Zacharopoulos.
\newblock The initial-to-final-state inverse problem with critically-singular potentials.
\newblock Preprint, {arXiv}:2602.12122 [math.{AP}] (2026), 2026.

\bibitem{Caro2011}
Pedro Caro.
\newblock On an inverse problem in electromagnetism with local data: stability and uniqueness.
\newblock {\em Inverse Problems and Imaging}, 5(2):297–322, 2011.

\bibitem{Caro2013}
Pedro Caro, Andoni García, and Juan~Manuel Reyes.
\newblock Stability of the calderón problem for less regular conductivities.
\newblock {\em Journal of Differential Equations}, 254(2):469–492, January 2013.

\bibitem{zbMATH07801151}
Pedro Caro and Alberto Ruiz.
\newblock An inverse problem for data-driven prediction in quantum mechanics.
\newblock {\em J. Math. Phys.}, 65(1):28, 2024.
\newblock Id/No 011505.

\bibitem{arXiv:2512.04796}
Pedro Caro and Alberto Ruiz.
\newblock The initial-to-final-state inverse problem with unbounded potentials and {Strichartz} estimates.
\newblock Preprint, {arXiv}:2512.04796 [math.{AP}] (2025), 2025.

\bibitem{HormII}
Lars H\"ormander.
\newblock {\em The analysis of linear partial differential operators. {II}}, volume 257 of {\em Grundlehren der mathematischen Wissenschaften [Fundamental Principles of Mathematical Sciences]}.
\newblock Springer-Verlag, Berlin, 1983.
\newblock Differential operators with constant coefficients.

\bibitem{zbMATH02204588}
Alexandru~D. Ionescu and Carlos~E. Kenig.
\newblock Well-posedness and local smoothing of solutions of {Schr{\"o}dinger} equations.
\newblock {\em Math. Res. Lett.}, 12(2-3):193--205, 2005.

\bibitem{Kian2019}
Yavar Kian and Eric Soccorsi.
\newblock H\"{o}lder stably determining the time-dependent electromagnetic potential of the schr\"{o}dinger equation.
\newblock {\em SIAM Journal on Mathematical Analysis}, 51(2):627–647, January 2019.

\bibitem{Lai2024}
Ru-Yu Lai, Xuezhu Lu, and Ting Zhou.
\newblock Partial data inverse problems for the nonlinear time-dependent schr\"{o}dinger equation.
\newblock {\em SIAM Journal on Mathematical Analysis}, 56(4):4712–4741, 2024.

\bibitem{Rudin}
W.~Rudin.
\newblock {\em Real and complex analysis/}.
\newblock McGraw-Hill,, New York:, 3rd edition, 1987.

\bibitem{Stefanov1990}
Plamen Stefanov.
\newblock Stability of the inverse problem in potential scattering at fixed energy.
\newblock {\em Annales de l’Institut Fourier}, 40(4):867–884, 1990.

\end{thebibliography}
\bibliographystyle{plain}

\end{document}